\documentclass[11pt]{amsart}
\usepackage{amsmath}
\usepackage{amssymb}
\usepackage[active]{srcltx}
\usepackage[lite]{amsrefs}
\usepackage{amscd}
\usepackage{enumerate}

\newtheorem{bigtheorem}{Theorem}
\newtheorem{theorem}{Theorem}[section]

\newtheorem{defn}[theorem]{Definition}

\newtheorem{prop}[theorem]{Proposition}
\newtheorem{lemma}[theorem]{Lemma}
\newtheorem{fact}[theorem]{Fact}
\newtheorem{cor}[theorem]{Corollary}
\newtheorem{claim}[theorem]{Claim}
\newtheorem*{claim*}{Claim}
\newtheorem*{open-question*}{An open Question}
\newtheorem*{question*}{Question}
\theoremstyle{remark}

\newtheorem*{note*}{Note}

\theoremstyle{definition}

\renewcommand{\geq}{\geqslant}
\renewcommand{\leq}{\leqslant}

\newcommand{\CB}{\mathcal B}

\newcommand{\sub}{\subseteq}
\newcommand{\si}{\sigma}
\newcommand{\la}{\langle}
\newcommand{\ra}{\rangle}
\newcommand{\tp}{\mathrm{tp}}

\newcommand{\be}{\begin{enumerate}}
\newcommand{\ee}{\end{enumerate}}

\newcommand{\CI}{\mathcal I}

\newcommand{\CN}{{\mathcal N}}

\newcommand{\CM}{{\mathcal M}}

\newcommand{\CF}{\mathcal F}

\newcommand\vlabel{\label}

\newcommand{\RR}{{\mathbb R}}

\newcommand{\CU}{\mathcal U}

\newcommand{\scht}{:}

\DeclareMathOperator{\dcl}{dcl}
\DeclareMathOperator{\shcl}{shcl}
\DeclareMathOperator{\lgdim}{lgdim}
\DeclareMathOperator{\Cl}{Cl}
\DeclareMathOperator{\Sh}{Sh}
\DeclareMathOperator{\bd}{bd}
\DeclareMathOperator{\fr}{fr}

\author{Pantelis Eleftheriou}
\address{University of Waterloo}
\email{pelefthe@uwaterloo.ca}

\author{Ya'acov Peterzil}
\address{University of Haifa}
\email{kobi@math.haifa.ac.il}

\author{Janak Ramakrishnan}
\address{Universidade de Lisboa}
\email{janak@janak.org}
\thanks{The authors thank M\'ario Edmundo and the FCT grant PTDC/MAT/101740/2008 and M\'ario Edmundo for bringing them together in Lisbon, where this work was begun.}

\title{Interpretable groups are definable}


\begin{document}

\begin{abstract}
We prove that in an arbitrary o-minimal structure, every interpretable group is
definably isomorphic to a definable one. We also prove that every definable group
lives in a cartesian product of one-dimensional definable group-intervals (or
one-dimensional definable groups). We discuss the general open question of elimination of imaginaries in an o-minimal structure.
\end{abstract}

 \maketitle

\section{Introduction}
Elimination of Imaginaries, namely the ability to associate a definable set to
every quotient of another definable set by a definable equivalence relation,  plays a major role in
modern model theory. In the study of o-minimal structures this issue is often
avoided by making the auxiliary assumption that the structure expands an ordered
group. Indeed, this assumption resolves the matter because o-minimal expansions of
ordered groups  eliminate imaginaries in a very strong form (see \cite[Proposition
6.1.2]{vdd}), namely every $A$-definable equivalence relation has an $A$-definable
set of representatives, after naming a non-zero element of the group. In particular,
the structure has definable Skolem functions. Since most interesting examples of
o-minimal structures do expand ordered groups and even ordered fields, this
assumption seems reasonable for most purposes.

Recently, following the work on Pillay's Conjecture, it was shown in \cite[Corollary
8.7]{HP} that even when starting with a definable group $G$ in expansions of real
closed fields, the group $G/G^{00}$ is definable in an o-minimal structure over
$\RR$ which is not known \emph{a priori} to expand an ordered group.

 In this paper we are concerned with the two issues raised above. First, we
 are interested to know to what extent o-minimal structures in general eliminate
 imaginaries. Second, we show that in many cases the assumption that the underlying structure
 expands a one-dimensional group is indeed harmless because this group is already
 definable in our structure (actually, we might need several different such groups).

Let us be more precise now. Definable equivalence relations can be treated either
within the many-sorted structure $\CM^{eq}$, or explicitly, as definable objects in
$\CM$. In order to apply o-minimality, we mostly work in the latter context, so we
clarify some definitions.  We assume we work in an arbitrary o-minimal structure.
\begin{defn} Let $X, Y$ be definable sets, $E_1,E_2$ two definable equivalence relations
on $X$ and $Y$, respectively. A function $f:X/E_1\to Y/E_2$ is called {\em
definable} if the set $\{\la x, y\ra\in X\times Y:f([x])=[y]\}$ is definable.
\end{defn}

\begin{defn}
Let $E$ be a definable equivalence relation on a definable set $X$, where both $X$
and $E$ are defined over a parameter set $A$. We say that the quotient \emph{$X/E$
can be eliminated over $A$} if there exists an $A$-definable injective
map $f:X/E\to M^k$, for some $k$. We say in this case that $f$ {\em eliminates $X/E$
over $A$}.
\end{defn}

It was already observed in \cite{Pillay1} that quotients cannot in general be
eliminated in o-minimal structures, over arbitrary parameter sets. Indeed, consider
the expansion of the ordered real numbers by the equivalence relation on $\RR^2$
given by: $\la x,y\ra \sim \la z,w\ra$ if and only if $x-y=z-w$. This quotient
cannot be eliminated over $\emptyset$. However, once we name any element $a$, the map
$f(\la y,z\ra/\sim)=a+y-z$ is definable and eliminates this quotient (over $a$).

It is therefore reasonable to ask:

\begin{question*}
Given an o-minimal structure $\CM$ and a definable equivalence relation $E$ on a
definable set $X$, both defined over a parameter set $A$, is there a definable map
which eliminates $X/E$, possibly over some $B\supseteq A$?
\end{question*}

We give a positive answer to this question when $\dim(X/E)=1$ (see Corollary
\ref{EI1}), but the general question remains open.

\begin{defn}\label{def-int-gp} An
\emph{interpretable group} is a group whose universe is a quotient $X/E$ of a
definable set $X$ by a definable equivalence relation $E$, and whose group operation
is a definable map.
\end{defn}

As we will see below, we prove in this paper that interpretable groups can be
eliminated.

\subsection{Groups in o-minimal structures}

The analysis of definable groups in o-minimal structures depends to a large extent
on a theorem of Pillay, \cite{Pillay2}, about the existence of a definable basis for
a group topology. The theorem holds for definable groups, but until now it was not
clear how to treat interpretable groups.  In \cite[Proposition 7.2]{ed-solv},
Edmundo was able to circumvent part of this problem by showing that if a group $G$
is already definable in an o-minimal structure $\CM$ then $\CM$ has elimination of
imaginaries for definable subsets of $G$. This makes it possible to handle
interpretable groups which are obtained as quotients of one definable group by
another. But interpretable groups in general remained out of reach (see Appendix in
\cite{HP} for the technical difficulties which may arise).

The following theorem, which we prove here, reduces the study of interpretable groups to definable ones.

\begin{bigtheorem}\label{intdefthm}
Every interpretable group is definably isomorphic to a definable group.
\end{bigtheorem}

In order to describe the main ideas of the proof we need to return to the second
problem mentioned at the beginning of this introduction.

\subsection{Group-intervals}

As we already remarked, it is often convenient to assume that the o-minimal
structure expands an ordered group. Beyond the experimental fact that most examples
have this property, there is another justification for this assumption, related to
the Trichotomy Theorem (\cite{pest-tri}), as we discuss next.

Recall that a point $x\in M$ is called {\em nontrivial} if there exist open nonempty
intervals $I,J\sub M$, with $x\in I$, and a definable continuous function $F:I\times
J\to M$ such that $F$ is continuous and strictly monotone in each variable
separately (the original definition required $I=J$ but it is easy to see that the
two are equivalent).

The Trichotomy Theorem implies that if $x\in M$ is non-trivial then there exists an
open interval $I'\ni x$ that can be endowed with a definable partial group operation
$+$, making $I'$ into a \emph{group-interval} (a technical definition will appear in Definition \ref{def-gp-interval} below, but
for now, we may think of a group-interval as an open interval $(-a,a)$ in an ordered
divisible abelian group, endowed with the partial group operation). The definition
of the group operation on $I'$ may require additional parameters.

%and the size of the interval could be small relative to $M$.

Consider for example the expansion of the ordered real numbers by the ternary
operation $x+y-z$, defined for all $x, y, z$ with $|x-y|, |y-z|, |x-z|\leq 1$. In
this structure and in elementarily equivalent ones every point $a$ is non-trivial
and contained in an $a$-definable group-interval. Note however that the
group-intervals can be `far apart', meaning that there are no definable bijections
between them.

In our current paper we propose a systematic treatment  of the
group-intervals which arise from the Trichotomy Theorem and
suggest a technique of ``stretching'' these intervals as much as
possible. We call an interval {\em group-short} (Definition \ref{def-gp-short}) if it can be be
written as a finite union of points and open intervals, each of
which endowed with the structure of a group-interval. After
\cite{el-sbd}, we develop a pre-geometry based on the closure
relation: $a\in cl(A)$ if there is a gp-short interval containing
$a$ whose endpoints are in $\dcl(A)$. Our main theorem here (Theorem \ref{mainthm}) is:

\begin{bigtheorem}\label{bingpthm}
Let $I,J\sub M$ be open intervals and assume that there exists a definable
$F:I\times J\to M$ which is continuous and strictly monotone in each variable. Then
either $I$ or $J$ (but possibly not both) is group-short.
\end{bigtheorem}

\subsection{} Let us now sketch the proof
of Theorem \ref{intdefthm}. We start with an interpretable group $G$. Using
elimination of one-dimensional quotients (Corollary \ref{EI1}) and Theorem \ref{bingpthm}, we prove
first that every one-dimensional subset of $G$ is group-short (Theorem \ref{inter-group1}).
We also show, in Proposition \ref{observ-eq}, that there are definable maps $f^i:G\to M$, $i=1,\ldots,
k$ and a definable set $X\sub \Pi_{i=1}^kf^i(G)$ such that $G$
is in definable bijection with $X/E$ for some definable
equivalence relation $E$. Our final goal is to prove that each
definable set $f^i(G)$ is a finite union of group-short intervals
(in which case we can eliminate $X/E$).

To achieve that, we endow $G$ with a group topology with a definable
basis. This is done by identifying a neighborhood of a generic
point in $G$ with an open subset of $M^{\dim(G)}$. Just as with
definable groups, we can use the distinction between definably
compact interpretable groups and those which are not definably
compact.  In the first case, we prove in Theorem \ref{SDC} definable choice for
definable subsets of $G$ using Edmundo's ideas \cite{ed-solv}. As
a result, it follows that each $f^i(G)$ is group-short. In the
general case, we use induction on dimension, together with the
standard analysis of groups definable in o-minimal structures as quotients
of semisimple groups, torsion-free abelian groups, etc. This finishes our final goal and the proof of Theorem \ref{intdefthm}.

At the end
of the argument we show not only that $G$ is in definable
bijection with a definable group, but also prove:

\begin{bigtheorem}
If $G$ is a definable group then there is a definable injection
$f:G \to \Pi_{i=1}^k J_i$, where each $J_i\sub M$ is a definable
group-interval.

There are also definable one-dimensional groups $H_1,\ldots, H_k$
and a definable set-injective map $h:G \to \Pi_{i=1}^k H_i$ (with
no assumed connection between the group operations of $G$ and of
the $H_i$'s).
\end{bigtheorem}

Note that the group-intervals (or the groups) in the above result
are not assumed to be orthogonal to each other, namely,  there
could be definable maps between some of them. However, the theorem
might help in reducing problems about definable groups, such as
Pillay's Conjecture, to structures which expand ordered groups, or
at least group-intervals. As a first attempt, it would be
interesting to see if one can prove, using Theorem 3, an analogue
of the Edmundo-Otero theorem, \cite{EO}, on the number of torsion
points in definably compact abelian groups in arbitrary o-minimal
structures.

Theorem 3 answers positively a question which Hrushovski asked the
second author in past correspondence.
\\

{\bf On the structure of the paper:} In Section 2 we recall the Marker-Steinhorn
theorem and apply it for our purposes. In Sections 3 and 4 we study various
properties of  group-intervals and then use these, in Section 5, to develop the
pre-geometry of the short closure. In Section 6 we prove Theorem 2 and in Section 7
we discuss quotients and their various properties. Finally, in Section 8 we analyze
interpretable groups and prove Theorem 1 and Theorem 3.

\section{Model theoretic preliminaries}

Fix $\CM=\la M,<,\ldots\ra$ an arbitrary (dense) o-minimal structure, with or
without endpoints. The following observation is is easy:
\begin{fact}\vlabel{observation}
Assume that  $M_1\neq \emptyset$ is a subset of $M$ with the following properties:

(i) $\dcl_{\CM}(M_1)=M_1$.

(ii) The restriction of $<$ to $M_1$ is a dense linear ordering.

(iii) If $M_1$ has a maximum (minimum) point then $\CM$ has a
maximum (minimum) point.

Then $\CM_1=\la M_1,<\ldots\ra$ is an elementary substructure of
$\CM$.
\end{fact}

\begin{prop} \vlabel{prelim} Assume that for all $a,b,c\in M$,  there is no definable
bijection between intervals of the form $(a,b)$ and $(c,+\infty)$,  and there is
also no definable bijection between intervals of the form $(-\infty,a)$ and
$(b,+\infty)$.  Let $\CM\prec\CN$ and let $M_1=\{x\in N:\exists m\in M\,\, m\geq
x\, \}$ be the ``downward closure'' of $\CM$ in $\CN$. Then $M_1$ is a substructure
of $\CN$, and
\begin{enumerate}
\item  $\CM_1\prec \CN$.

\item If $X\sub N^k$ is an $\CN$-definable set then $X\cap M_1^k$ is
$\CM_1$-definable.

\end{enumerate}

\end{prop}
\proof (1) By the choice of $M_1$ as the downward closure of an elementary
substructure, $M_1$ satisfies (ii) and (iii) of \ref{observation}. It is therefore
sufficient to prove that $\dcl_{\CN}(M_1)=M_1$. The proof is similar to \cite[Lemma
2.3]{pet-sbd}

As in \cite{pet-sbd}, induction allows us to treat only the case of $b\in
\dcl_{\CN}(a)$, for $a\in M_1$. We must show that $b\in M_1$, so it is sufficient to
find an element $m\in M$, with $b\leq m$. If $b\in \dcl_{\CN}(\emptyset)$ then it is
already in $M$ so we are done. Otherwise, there is a $\emptyset$-definable, continuous,
strictly monotone function $f:(a_1,a_2)\to M$, for $a_1,a_2\in M\cup\{\pm \infty\}$,
such that $a\in (a_1,a_2)$ and $b=f(a)$.

Assume first that $f$ is strictly increasing on $(a_1,a_2)$ and consider two cases:
If $a_2=+\infty$ then, by our construction of $M_1$,  there exists $m\in
(a_1,+\infty)$ with $a\leq m$. Hence $b=f(a)\leq f(m)\in M$. If $a_2\in M$ then, by
our assumptions, the limit $\ell=\lim_{t\to a_2^-}$ is in $M$ so we have $b\leq
\ell$.

Assume now that $f$ is strictly decreasing. Then, by our assumptions on $\CM$,  the limit $\ell=\lim_{t\to a_1^+}f(t)$ is not
$+\infty$. It follows that $\ell\in M$ and by monotonicity, $b\leq \ell$.
 We therefore
showed that $\dcl_{\CN}(M_1)=M_1$.

(2) Since $\CM_1$ is convex in $\CN$ it is clearly Dedekind
complete in $\CN$ and hence we can apply the Marker-Steinhorn
theorem, \cite{ms}, on definability of types which says exactly
what we need. \qed

\section{group-intervals}
\begin{defn}\label{def-gp-interval}
 By a {\em positive group-interval} $I=\la (0,a),0,+,<\ra$ we mean an open interval with a
binary partial continuous operation $+: I^2\to I$, such that

\renewcommand{\labelenumi}{(\roman{enumi})}
\begin{enumerate}
\item $x+y=y+x$ (when defined), $(x+y)+z=x+(y+z)$ when defined, and
$x<y \rightarrow x+z<y+z$ when defined.
\item For every $x\in I$ the domain of $y\mapsto x+y$ is an
interval of the form $(0,r(x))$.
\item For every $x\in I$, we have $\lim_{x'\to 0}x'+x=x$ (this
replaces the statement $0+x=x$) and $\lim_{x'\to r(x)}x+x'=a$
(this replaces $x+r(x)=a$).
\end{enumerate}
\renewcommand{\labelenumi}{(\arabic{enumi})}

We say that $I$ is a {\em bounded positive group-interval} if the operation $+$ is
only partial. Otherwise we say that it is {\em unbounded} (in which case the
interval is actually  a semigroup).

 We similarly define the notion of a {\em negative
group-interval} $\la (a,0),+,<\ra$ and also a {\em group-interval} $\la (-a,a),+,<\ra$ (in this case we
also require that for every $x\in (-a,a)$ there exists a group inverse). We say that
an open interval $I$ is a {\em generalized group-interval} if it is one of the above
possibilities.

\end{defn}

Our use of the symbols $0,a,-a$ is only suggestive. The endpoints of the
interval can be arbitrary elements in $M\cup\{\pm \infty\}$,  so when we write
that an interval $(b,c)$ is, say, a bounded  group-interval, we think
of the elements $b$ and $c$ as $a$ and $-a$, respectively, from the definition.

\begin{note*}
If the interval $(a,b)$ can be endowed with
a definable $+$ which makes it into a generalized group-interval
then there is an $ab$-definable family of such operations (we just
take the operation $+$ and vary the parameters which defined it,
and further require the domain to be $(a,b)$ and the operation to
satisfy (i), (ii) and (iii) from the definition).
\end{note*}

 The following is easy to verify:
\begin{fact}\vlabel{easy1}

(i) If $(a,b)$ can be endowed with the structure of a bounded
group-interval then we can also endow it with a structure of a
bounded positive group-interval (making $a$ into $0$).

(ii) Conversely, if $(a,b)$ can be endowed with a structure of a bounded positive
group-interval then it can also be endowed with the structure of a bounded group-interval.

(iii) If $I$ is a generalized group-interval then any nonempty
open subinterval of $I$ can be endowed with the structure of a
generalized group-interval.
\end{fact}

\begin{theorem}\vlabel{groups} Assume that $\CM$ is an  o-minimal structure
 and let $I_t=(a_0,a_t)$, $t\in T$, be a
definable family of intervals, all with the same left endpoint. Let
$I=(a_0,a)=\bigcup_t I_t$. If each interval $I_t$ can be endowed with the structure
of a generalized group-interval then there is $a_1$, $a_0\leq a_1<a$ such that
$(a_1,a)$ admits the structure of a generalized group-interval.
\end{theorem}
\proof  First note that if there exists some $a_1\in [a_0,a)$ and a definable
continuous injection sending $(a_1,a)$ onto a subinterval $(a_2,a_3)\sub (a_0,a)$,
with $a_2<a_3<a$, then $(a_2,a_3)$ is contained in one of the intervals $I_t$ and
hence, by \ref{easy1}(iii), it inherits a structure of a generalized group-interval
itself. Clearly then $(a_1,a)$ can also be endowed with such a structure. We assume
then that there is no such definable injection in $\CM$.

 Consider now the structure $\CI$ which $\CM$ induces on
the interval $I=(a_0,a)$. By that we mean that the $\emptyset$-definable sets in
$\CI$ are the intersection of $\emptyset$-definable subsets of $M^n$ with $I^n$. By
\cite[Lemma 2.3]{pest-tri}, every $\CM$-definable subset of $I^n$ is definable in
$\CI$ (the result is proved for closed intervals but the result for open intervals
immediately follows). The points $a_0$ and $a$ are now identified with $-\infty$ and
$+\infty$ in the sense of $\CI$, respectively. We may assume from now on that
$\CM=\CI$.

Our above assumptions on $\CI$ translate to the fact that $\CM$
satisfies the assumptions of Proposition \ref{prelim}. Namely,
that there are no $-\infty\leq a_1,a_2<+\infty$ and $a_3\in M$ for
which $(a_1,+\infty)$ is in definable bijection with an interval
of the form $(a_2,a_3)$.

Using our Note above, we may assume that there is a $\emptyset$-definable family of
(partial) operations $+_t:I_t\times I_t\to I_t$ making each $I_t$ into a generalized
group-interval. Indeed, to see that, we use the note to  ``blow up'' each $I_t$ to a
$t$-definable family of group-intervals $\{I_{s,t}=\la I_t, +_{s,t}\ra :s\in S_t\}$,
all of them with domain $I_t$. By compactness, we can show that as we vary $t\in T$
the family of $S_t$'s and $+_{s,t}$ can be given uniformly. We now replace the
original family $\{I_t\}$, with the family $\{I_{s,t}:t\in T\,\, \& \, s\in S_t\}$,
on which the group operations are given uniformly. Furthermore, we may assume that
all intervals are either positive group-intervals, negative group-intervals, or
group-intervals {\em uniformly} (we partition the family into the various sets and
choose one whose union is still of the form $(-\infty,+\infty)$). For simplicity we
still denote the intervals by $I_t$ and the parameter set by $T$.

We first consider the case where each $I_t$ is a positive group-interval (bounded or unbounded).

  Each
interval $I_t=(-\infty, a(t))$ is a positive group-interval
(recall that in $\CM$ the point $-\infty$ plays the role of $0$).
Furthermore, we have $\bigcup_{t\in T} I_t=(-\infty,\infty)$.
Consider now a sufficiently saturated elementary extension $\CN$
of $\CM$ and take $a'<+\infty$ in $N$ such that $a'>m$ for all
$m\in M$. By our assumptions, there is $t_0\in T(\CN)$ such that
$(-\infty,a')\sub I_{t_0}$ and therefore there is a  positive
group-interval operation $+_{t_0}$ on the interval $(-\infty,a')$,
which is definable in $\CN$.

We now let $\CM_1$ be the downward closure of $M$ in $\CN$ as in
Proposition \ref{prelim}. By the same proposition, the
intersection of the graph of $+_{t_0}$ with $M_1^3$, call it $G$,
is a definable set in the structure $\CM_1$.

Let's see first that in $\CM_1$, the set $G$ is the graph of a
positive group-interval operation on $(-\infty,\infty)$ (with
$-\infty$ playing the role of $0$).

(1) $G$ is the graph of a partial function from $M_1^2$ into
$M_1$: this is clear since for every $(x,y)\in N^2$ there is at
most one $z\in N$ such that $(x,y,z)\in G$. Call it $+_G$.

(2) $+_G$ is continuous, since the order topology of $M_1$ is the subspace topology of
$N$ and $M_1$ is convex in $N$.

(3) $+_G$ is associative and commutative when defined, as inherited from $\CN$.

(4) $+_G$ respects order: again, inherited from $\CN$.

(5) For every $x\in (-\infty,\infty)$, the domain of $y\mapsto x+_G y$ is a convex
set in $M_1$ of the form $(-\infty,r_G(x))$: Indeed, the domain of $y\mapsto
x+_{t_0}y$ in $\CN$ is an interval $(-\infty ,r_{t_0}(x))$. Hence,  the domain of
$y\mapsto x+_G y$ is the intersection of $(-\infty ,r_{t_0}(x))$ with $M_1$. Since
$M_1$ is closed downwards in $\CN$, this intersection is $(-\infty, r_G(x))$, where
$r_G(x)=+\infty$ if $r_{t_0}(x)$ is greater than all elements of $M_1$ and otherwise
it is some element of $M_1$.

(6) Consider $\lim_{x'\to 0}x'+_G x$. Since this limit was $-\infty$ in $\CN$ (i.e.
$0$ in the original structure), it remains so in $\CM_1$, because $M_1$ was
downwards closed in $\CN$. It is left to see that  $\lim_{x'\to r_G(x')}x+_G
x'=\infty$ (i.e. $a$ in the original structure). This follows from the fact that for
every $t$ we have $\lim_{x'\to r_{t}(x')}x+_{t} x'=a(t)$, and $\sup_t a(t)=\infty$.

We therefore showed that $+_G$ makes $(-\infty,+\infty)$ a
positive group-interval in the structure $\CM_1$.

Since $\CM\prec \CM_1$ we can now write down the (first-order) properties which make
$+_G$ into an operation of a positive group-interval in $M_1$ and obtain an
operation $+$ on $M$, which is definable in $\CM$. This completes the case where
each $I_t$ is a positive group-interval.

Assume now that each $I_t$ is a group-interval. If each $I_t$ is bounded then, as we
noted earlier we can transform it into a positive bounded group-interval and finish
as above. If $I_t$ is unbounded then $a_0=-\infty_t$ and $a(t)=+\infty_t$. We can
now fix some $a_1\in (a_0,a)$ and restrict our attention to those $t$'s for which
$a_1\in I_t$. For each such $t$ we can endow $(a_1,a(t))$ with the structure of an
unbounded positive group-interval, and then finish as above.

Finally, if each $I_t$ is a negative group-interval (so
$I_t=(a_0,a(t))=(-\infty,a(t))$), then we can again assume that there is an $a_1$
which belongs to all $I_t$, and replace each $I_t$ with the interval $(a_1,a(t))$,
endowed with the structure of a bounded positive group-interval. This ends the proof
of the theorem.
 \qed
\\

\noindent {\bf Note}: We don't claim that the operation $+_G$ that we obtain in
$\CM_1$ belongs to the family $\{+_t:t\in T\}$ that we started with. E.g., in the
structure $\la \mathbb R,<,+ \ra$, take $+_t$ to be the restriction of the usual $+$
in $\mathbb R$ to an interval $I_t=(0,t)$. Each $I_t$ is a bounded positive
group-interval but the union $(0,+\infty)$ can only be endowed with the structure of
an unbounded positive group-interval.
\\

We end this section with an observation about group-intervals and definable groups.
\begin{lemma}\vlabel{intervals-groups}
Let $\la I,+\ra$ be a generalized group-interval. Then there
exists a definable one-dimensional group $\la H,\oplus\ra$ and a
definable $\si:I\to H$, such that $\si(x+y)=\si(x)\oplus \si(y)$,
when $x+y$ is defined. Said differently, every generalized group-interval can be embedded into a definable one-dimensional group.

If $I$ is a bounded generalized group-interval the $H$ is
definably compact and if $I$ is unbounded then $H$ is linearly
ordered.
\end{lemma}
\proof Assume that $I=(0,\infty)$ is an unbounded positive group-interval. Then we
let $H=I\times \{-1\}\cup\{0\}\cup I\times\{+1\}$ (with $-1,0,+1$ suggestive symbols
for elements in $M$). We define $a\oplus 0=x$ for every $a\in H$ and define $\la
x,i\ra\oplus \la y,j\ra$ to be $\la x+y,i\ra$ if $i=j$. If $i\neq j$ and $x<y$ we
let $\la x,i\ra\oplus \la y,j\ra=\la z,j\ra$, with $z\in I$ the unique element such
that $x+z=y$. if $y<x$ then $\la x,i\ra\oplus \ra y,j\ra=\la z,i\ra$, with $z$ the
unique element in $I$ such that $y+z=x$. The group $H$ we obtain is linearly ordered
and torsion-free. Obviously $I$ is embedded in $H$.

Assume now that $I=(0,a)$ is a positive bounded group-interval and let $a/2\in I$
denote the unique element  in $I$ such that $\lim_{t\to a/2}t/2+t/2=a$. We consider
$H$ the half-open interval $[0,a/2)$ with addition ``modulo $a/2$''. Namely, for
$x,y\in [0,a/2)$,
$$x\oplus y=\left \{\begin{array}{ll}
  x+y & \mbox{ if } x+y\in [0,a/2) \\
  x+y-a/2 &  \mbox{ if } x+y\geq a/2
\end{array} \right.
$$

The group $H$ is a one-dimensional definably compact group. To see that $I$ is
embedded in $H$, consider the map $x\mapsto x/4$ sending $I$ into $(0,a/4)$ (by
$x/4$ we mean the unique element $y\in I$ such that $y+y+y+y=x$. It is easy to check
that this is an embedding of $I$ into $H$.\qed

\section{Gp-short and gp-long intervals}

\subsection{Definitions and basic properties}

We assume here that $\CM$ is an arbitrary sufficiently saturated o-minimal structure.

\begin{defn}\label{def-gp-short} An interval $I\sub M$ is called a {\em group-short (gp-short) interval} if it can
be written as a finite disjoint union of points and open intervals, each of which
can be endowed with the structure of a generalized group-interval. An interval which
is not group-short is called a {\em gp-long interval}.

Although there is no global notion of distance in $M$, in abuse of notation we
say that the {\em distance between $a,b\in M$ is gp-short} if either $a=b$, or  the
interval $(a,b)$ (or $(b,a)$) is gp-short. Otherwise, we say that this distance is
gp-long.
\end{defn}

Note that points, being trivial closed intervals, are gp-short.

\begin{defn} A definable set $S\sub M^n$ is called a {\em gp-short set} if there
are gp-short intervals $I_1,\ldots, I_k$ such that $S$ is in definable bijection
with a subset of $\Pi_j I_j$.
\end{defn}

\noindent{\bf Note } \begin{enumerate} \item It is not hard to see that the above
definition coincides with the previous one in the case of intervals. Namely, if an
interval $I$ is also a gp-short set then it can be written as a finite union of
points and open group-intervals.  \item As before, if $(a,b)$ is a gp-short interval
then it can be endowed with an $ab$-definable family of subintervals, with
operations on them, witnessing the fact that $(a,b)$ is gp-short. Indeed, we start
with particular parameterically definable such witnesses and let the parameters
(including the end points of the sub-intervals) vary.

\item If $S$ is a finite union of gp-short intervals then it is in definable
bijection with a definable subset of their cartesian product, after possibly naming
finitely many points. For example, the disjoint union of $I$ and  $J$ is in
definable bijection with the set $$(I\times \{b\}\sqcup \{a\}\times J) \cup
\{\la a',b\ra\},$$ for any distinct $a,a'\in I$ and $b\in J$.

It follows that a finite union of gp-short sets  is  a gp-short set.

\item It is of course possible that the only gp-short sets in $\CM$ are finite,
namely  there are no definable generalized group-intervals in $\CM$. The Trichotomy
Theorem, \cite{pest-tri}, tells us that in this case the definable closure is
trivial and every point in $\CM$ is trivial. This is equivalent to the fact
(\cite{MRS}) that $\CM$ has quantifier elimination down to $\emptyset$-definable
binary relations.

\item Clearly, if $I$ is a gp-short interval and $f:I\to M$ is a definable continuous
injection then $f(I)$ is also a gp-short interval.
\end{enumerate}

\begin{fact}\vlabel{SDC-short} If $I_1,\ldots, I_k$ are gp-short intervals then, after
fixing finitely many parameters $A$,  the product $X=\Pi_j I_j$ has strong definable
choice. Namely, if $\{S_t:t\in T\}$ is a $B$-definable family of subsets of $X$ then
there is an $AB$-definable function $\sigma:T\to X$ such that for every $t\in T$, we
have $\sigma(t)\in S_t$ and if $S_{t_1}=S_{t_2}$ then $\si(t_1)=\si(t_2)$.
\end{fact}
\proof We write each $I_j$ as a finite union of points and generalized group-intervals (possibly over extra parameters), and then repeat standard proof of
definable choice in expansions of ordered groups (see \cite{vdd}), using the group
operations on each interval.\qed

\begin{fact} \vlabel{gp-short2} Assume that $S\sub M^n$ is a gp-short set and $f:S\to M^k$ is a
definable map. Then $f(S)$ is also a gp-short set.
\end{fact}
\proof By definition, we may assume that $S\sub \Pi_j I_j$, for $I_1,\ldots, I_k$
gp-short intervals. By Fact \ref{SDC-short}, there is a definable set $X_0\sub S$
such that $f|X_0$ is a bijection between $X_0$ and $f(X_0)=f(S)$. By definition,
$f(S)$ is a gp-short set.\qed

We now collect a list of important properties.
\begin{fact}\vlabel{int-fact0}
Let $\{I_t:t\in T\}$ be a definable family of intervals. Then
\renewcommand{\labelenumi}{(\roman{enumi})}
\begin{enumerate}
\item The set of all $t\in T$ such that $I_t$ is gp-long is type-definable \item The
set of all $t\in T$ such that $I_t$ is gp-short is $\bigvee$-definable. \item If
$\bar a\in M^m$ is a tuple and the formula $\varphi(\bar x,\bar a)$ defines a
gp-short set in $M^n$, then there is a $\emptyset$-definable formula $\psi(\bar x)$ such
that $\psi(\bar a)$ holds and if $\psi(\bar b)$ holds then the set defined by
$\varphi(\bar x,\bar b)$ is gp-short.
\end{enumerate}
\renewcommand{\labelenumi}{(\arabic{enumi})}
\end{fact}
\proof For every natural number $K$ and for every definable family of $K$ functions
$F_1(x,y,\bar w_1),\ldots, F_K(x,y,\bar w_K)$, we can write a formula which says:
For every possible writing of $I_t$ as a union of $K$ intervals $I_1,\ldots, I_K$
and for every $\bar w_1,\ldots,\bar w_K$, it is not the case that $F_1(-,-,\bar
w_1),\ldots, F_K(-,-,\bar w_K)$ are operations making $I_1,\ldots, I_K$,
respectively, into group-intervals (here we need to go through the various possibilities of positive,
negative group-intervals etc). When varying over all possible $K$'s and all possible
families, we obtain a type-definable definition for the set of $t$'s for which $I_t$
is gp-long. The complement of this set is $\bigvee$-definable.

For (iii), note that if $(c,d)$ is a gp-short interval, then there is a formula
$\rho(c,d)$ saying that $(c,d)$ is the finite union of points and intervals, each of
which is a generalized group-interval. Let $\theta(\bar x,\bar x',\bar e)$ be an
$\bar e$-definable bijection between $\varphi(M^n,\bar a)$ and $\Pi_j I_j$ for some
gp-short intervals $I_j$'s. Let $\rho_j$ be the formula witnessing that $I_j$ is
gp-short for each $j=1,\ldots,m$. Then the desired formula $\psi(\bar y)$ says that
there exist parameters $\bar w$ such that $\theta(\bar x,\bar x',\bar w)$ defines a
bijection between $\varphi(M^n,\bar y)$ and $\Pi_j (c_j,d_j)$ for some gp-short
intervals $(c_j,d_j)$, witnessed by formulas $\rho_j$ for $j=1,\ldots,m$. \qed

\begin{theorem} \vlabel{gp-short}
Let
 $\{S_t:t\in T\}$ be a definable family of
gp-short, definably connected subsets of $M^n$ and assume that there is $a_0\in M^n$
such that for every $t\in T$, $a_0\in \Cl(S_t)$. Then $S=\bigcup_t S_t$ is a gp-short
set.
\end{theorem}

\newcommand{\R}{\mathbb R}
Before we prove the result we note that the requirement about $a_0$ is necessary:
Consider the structure on $\mathbb R$ with the restriction of the graph of $+$ to
all $a,b\in \R$ such that $|a-b|\leq 1$. In this structure (and in elementary
extensions) there is a group-interval around every point so the whole structure is a
union of gp-short intervals. However, the union (i.e. the universe) is not gp-short.

 \proof Let $\pi_i:M^n\to M$ be the projection onto the $i$-th coordinate.
  It is sufficient to show that each $\pi_i(S)$ is gp-short. Because $S_t$ is
definably connected, its projection $\pi_i(S_t)=I_t$ is an interval, which by Fact
\ref{gp-short2} is gp-short. Furthermore,  $\pi_i(a_0)\in \Cl(I_t)$. Hence, we may
assume from now on that $S_t=I_t$ is a gp-short interval in $M$ and $a_0\in
\Cl(I_t)$ for every $t$.  It is sufficient to prove that $I=\bigcup_t \Cl(I_t)$ is
gp-short, so by replacing $I_t$ with $\Cl(I_t)$ (still gp-short) we may assume that
$a_0\in I_t$ for all $t$. Let $I=(a,b)$.

\begin{claim*}
There is $b_1<b$ such that the interval $(b_1,b)$ is gp-short.
\end{claim*}

Since each $I_t$ is a gp-short interval the type $p(t)$, which says that $I_t$ is
gp-long (see Fact \ref{int-fact0}), is inconsistent. It follows that there exists a
fixed number $K$ such that every $I_t$ can be written as the union of at most $K$
generalized group-intervals and $K$ many points.

 We write
$I_t=I_{t,1}\cup\cdots\cup I_{t,K}\cup F_t$, such that each
$I_{t,i}=(a_{t,i},a_{t,i+1})$, $i=1,\ldots, K$, can be endowed with the structure of
a generalized group-interval,  and $F_t$ finite. The end points of the $I_t$'s are
 definable functions of $t$ and $b=\sup_t a_{t,K+1}$.

Let $a'=\sup_t \, a_{t,K}$ and assume first that $a'<b$.

 We can restrict ourselves to those $t\in T$ such that $a_{t,K+1}>a'$ and
consider each interval $(a',a_{t,K+1})$ as a sub-interval of $(a_{t,K},a_{t,K+1})$.
We already noted that $(a',a_{t,K+1})$  admits the structure of a generalized
group-interval. So, we write $a_t$ for $a_{t,K+1}$ and consider the family of all
generalized group-intervals $(a',a_t)$.
  By Theorem
\ref{groups} there exists $b_1<b$ such that $(b_1,b)$ admits an operation of a
generalized group-interval.

 Assume now that $a'=\sup\, a_{t,K}=b$. In this case, we can replace each $I_t$
 by $I_t'=I_{t,1}\cup\cdots\cup I_{t,K-1}$, and still have $\bigcup_t I'_t=I$, and
 finish by induction on $K$.

Just as we found $b_1$ above,  we can find $a_1>a$ such that $(a,a_1)$ admits a
definable generalized group-interval. Choose $t_1$ such that
$I_{t_1}\cap(a,a_1)\ne\emptyset$ and $t_2$ such that
$I_{t_2}\cap(b_1,b)\ne\emptyset$. Since $a_0\in I_{t_1}\cap I_{t_2}$, the union of
the two intervals is again an interval, containing $(a_1,b_1)$, and therefore
$(a_1,b_1)$ is gp-short.
 We can therefore conclude that $(a,b)$ is gp-short.\qed

As a corollary  we obtain:
\begin{cor}\vlabel{c-def} Let $(a,b)$ be an interval which is gp-short.
\begin{enumerate}
\item\label{int-around-c-item} Assume that $c\in (a,b)$. Then there exists
a $c$-definable interval $I\supset (a,b)$ such that $I$ is
gp-short (possibly witnessed by extra
parameters).
\item\label{int-to-a-item} There is an $a$-definable ($b$-definable) interval
$I\supseteq (a,b)$ which is gp-short (possibly witnessed by extra
parameters).

\end{enumerate}
\end{cor}
 \proof (1) By our earlier note, $(a,b)$ belongs to a $\emptyset$-definable family
 of gp-short intervals. Using the parameter $c$, we obtain a $c$-definable family
 of gp-short intervals, all containing $c$. By Theorem \ref{gp-short}, their union
 is gp-short (and clearly definable over $c$).

(2) We do the same, but now obtain an $a$-definable ($b$-definable) family of
intervals all with the same left-endpoint $a$ (right endpoint $b$). We now use
Theorem \ref{gp-short}.
 \qed

\begin{lemma}\vlabel{gp-short1} Let $\{S_t:t\in T\}$ be a definable family of
gp-short sets and assume that $T$ is a gp-short subset of $M^k$. Then the union
$S=\bigcup_{t\in T} S_t$ is gp-short.
\end{lemma}
\proof We may assume that $T$ is definably connected. By partitioning each $S_t$,
uniformly in $t$, into its definably connected components we can also assume that
each $S_t$ is definably connected. It is enough to see that the projection of $S$
onto each coordinate is gp-short. Let $\pi_1:M^n\to M$ be the projection onto the
first coordinate and let $I_t=\pi_1(S_t)$. By Fact \ref{gp-short2}, each $I_t$ is a
gp-short interval, so it is enough to prove that $\bigcup_{t\in T}I_t$ is gp-short.
Write $I_t=(a_t,b_t)$ with $a_t$ and $b_t$ definable functions of $t$. Again, after
a finite partition, we may assume that $t\mapsto a_t$ and $t\mapsto b_t$ are
continuous on $T$.

 Let $(a,b)=\bigcup_t I_t$, let $a_1=\sup_t \, a_t$ and $b_1=\inf_t \, b_t$.  The image of $T$ under
$t\mapsto a_t$ is an interval $I_1$ and the image of $T$ under  $t\mapsto b_t$ is
another interval $I_2$ (since $T$ is definably connected and the functions are
continuous). The interval $(a,b)$ equals, up to finitely many points,  $I_1\cup
[a_1,b_1]\cup I_2$.

If
 $a_1<b_1$ then the interval
$(a_1,b_1)$ is gp-short since it is contained in all $I_t$'s. By Fact
\ref{gp-short2},  $I_1$ and $I_2$ are gp-short, hence $(a,b)$ is gp-short. We
therefore showed that $\pi_1(S)$ is gp-short, and prove similarly that each
$\pi_i(S)$ is gp-short.\qed

\section{Short closure and gp-long dimension}

\subsection{Defining  short closure}
We follow here ideas from \cite{el-sbd}.

\begin{defn} For $a\in M$ and $A\sub M$ we say that {\em $a$ is in the short closure of
$A$}, written as $a\in \shcl(A)$, if either $a\in \dcl(A)$ or there is $b\in \dcl(A)$
such that the distance between $a$ and $b$ is gp-short. Equivalently, the closed
interval $[a,b]$ (or $[b,a]$) is gp-short.

\end{defn}

 Note that $\dcl(A)\sub \shcl(A)$.

 Clearly, if $\CM$ expands an ordered group then
$M=\shcl(\emptyset)$, so our definition really aims for those o-minimal structures
which do not expand ordered groups.

\begin{fact}\vlabel{shcl-fact} For every $a\in M$ and $A\sub M$,
$a\in \shcl(A)$ if and only if  there exists an $A$-definable, closed, gp-short
interval containing $a$.
\end{fact}
\proof The ``if'' direction is clear, so we only need to prove the ``only if''.
Assume that we have $[a,b]$ gp-short with $b\in \dcl(A)$. By Corollary
\ref{c-def}\eqref{int-to-a-item}, there is a $b$-definable gp-short interval $[c,b]$
which contains $[a,b]$, so $a\in [c,b]$.\qed

\begin{lemma}\vlabel{pregeoemtry}
The gp-short closure  is a pre-geometry. Namely:

(i) $A\sub \shcl(A)$.

(ii) $A\sub B$ $\Rightarrow$ $\shcl(A)\sub \shcl(B)$.

(iii) $\shcl(\shcl(A))=\shcl(A)$.

(iv) $\shcl(A)=\cup\{\shcl(B):B\sub A\mbox{ finite }\}.$

(v) (Exchange) $a\in \shcl(bA)\setminus \shcl(A) \,\rightarrow \, b\in \shcl(aA)$.
\end{lemma}

\proof (i) (ii) are  clear. (iii) Assume that $a_i\in \shcl(A)$ for $i=1,\ldots, n$.
By Fact \ref{shcl-fact}, for every $i$, there is a gp-short interval $I_i$
containing $a_i$.
 Assume now that $b\in \shcl(a_1,\ldots, a_n)$. We want to
show that $b\in \shcl(A)$. Let $S=I_1\times \cdots \times I_n$.

 By \ref{shcl-fact}, there is a gp-short interval $J_{\bar a}\ni b$, defined over
$a_1,\ldots, a_n$, which we may assume belongs to a $\emptyset$-definable family of
gp-short intervals. Consider the set of all intervals $J_{\bar s}$, for $\bar s\in
S$. By Fact \ref{gp-short1}, the union $J=\bigcup_{\bar s\in S} J_{\bar s}$ is
gp-short (and contains $b$).  Since $S$ is $A$-definable so is $J$.

(iv) is clear from the definition.

(v) Assume that $a\in \shcl(bA)\setminus \shcl(A).$ Then there is an
$Ab$-definable gp-short interval $[b_1,b_2]$ containing $a$. Since $a\notin
\shcl(A)$, it follows that $b_i\notin \dcl(A)$ for $i=1,2$, so Exchange for $\dcl$
implies that $b\in \dcl(b_iA)$. By \ref{c-def}, there is an $a$-definable gp-short
interval containing $[b_1,b_2]$ and hence $b_1,b_2\in \shcl(aA)$. By transitivity of
$\shcl$, proved in (i), we have $b\in \shcl(aA).$\qed

\subsection{Long dimension of tuples}

\begin{defn} A set $B\sub M $ is called {\em $\shcl$-independent over $A\sub M$ } if for every $a\in B$, we have
$a\notin \shcl(B\cup A\setminus \{a\})$. For $(a_1,\ldots,a_n)\in M^n$ and $A\sub M$
we let the  {\em gp-long dimension of $\bar a$ over $A$}, $\lgdim(\bar a/A)$, be the
maximal $m\leq n$ such that $\bar a$ contains a tuple of length $m$ which is
$\shcl$-independent over $A$.
\end{defn}

\begin{note*}
\leavevmode\begin{enumerate}
\item We have $\lgdim(a/A)\leq \dim(a/A)$.

\item Because the dimension is based on a pre-geometry we have  the dimension formula
$$\lgdim(a,b/A)=\lgdim(a/bA)+\lgdim(b/A).$$

\item If $\bar a,\bar b$ realize the same type over $A$ then $\lgdim(\bar
a/A)=\lgdim(\bar b/A)$.

\item If $\CM$ is an expansion of an ordered group then the whole universe is
gp-short and therefore $\lgdim(a/A)=0$ for every $a\in M$, $A\sub M$. On the other
end, it is possible that no group-intervals are definable in $\CM$. In this case,
$\shcl(A)=\dcl(A)$ and by the Trichotomy Theorem, \cite{pest-tri}, the resulting
pre-geometry is trivial.
\end{enumerate}

\end{note*}

\begin{defn}
For $I=(a,b)$ and $c\in I$, we say that $c$ is {\em long-central in $I$} if both
$(a,c)$ and $(c,b)$ are gp-long.
\end{defn}

\begin{fact}\vlabel{shcl1}

Let $A\subset M$ be smaller than the saturation of $M$.

\begin{enumerate}

\item If $I$ is a definable gp-long interval, then there is $a\in I$
such that $a\notin \shcl(A)$.

\item Let $a\in M^n$, $\lgdim(a/A)=k$, and $p(x)=\tp(a/A)$. Then for every
$B\supseteq A$ there exists $b\models p$ such that $\lgdim(b/B)=k$.

\item Let $I=(d_1,d_2)$ be a gp-long interval and $a\in I$
long-central. Given any $\bar b\in M^n$, there exist $c_1,c_2$, $d_1\leq
c_1<c_2\leq d_2$, such that $a$ is long-central in $(c_1,c_2)$ and $\lgdim(\bar
b/A)=\lgdim(\bar b/Ac_1c_2).$

\end{enumerate}
\end{fact}

 \proof

(1) Consider the type over $A$:
$$p(x):\{x\in I\} \cup \{x\notin (a_1,a_2):\,\, a_1,a_2\in \dcl(A)\,\, \& (a_1,a_2)
\mbox{ gp-short }\}$$ (note that in the definition of the type we are just going
over all $a_1,a_2\in \dcl(A)$ such that $(a_1,a_2)$ is gp-short. We don't claim any
uniformity here).

If $p(x)$ is inconsistent then $I$ is contained in a finite union of gp-short
intervals, which is impossible.

(2) We prove the result for $a\in M$, with $\lgdim(a)=1$. The case of $M^n$ is done
by induction. The set $p(M)$ can be written as the intersection of open intervals,
defined over $A$, which are necessarily gp-long. By (1), each such interval contains
a point $b\notin \shcl(B)$. By compactness we can find $b\models p$ with $b\notin
\shcl(B)$.

(3) Write $I=(d_1,d_2)$. Using (1), we first choose $c_1\in (d_1,a)$ such that
$c_1\notin \shcl(A\bar ba)$. In particular, $(c_1,a)$ is gp-long. Next, choose
$c_2\in (a,d_2)$ such that $c_2\notin \shcl(Ac_1\bar ba)$. It follows that
$\lgdim(c_1c_2/A\bar b)=2$ and therefore, by the dimension formula,  $\lgdim(\bar
b/Ac_1c_2)=\lgdim(\bar b/A)$.
 \qed

\subsection{Long dimension of definable sets}

\begin{defn} For $X\sub M^n$ definable over a small $A\sub M$, we let
$$\lgdim_A(X)=\max\{\lgdim(a/A):a\in X\}.$$

By Fact \ref{shcl1}(2), if $X$ is definable over $A$ and $A\sub B$ then
$\lgdim_B(X)=\lgdim_A(X)$, so we can let $\lgdim(X):=\lgdim_A(X)$ for any $A$ over which
$X$ is definable.

We say that $a\in X$ is {\em long-generic over $A$} if $\lgdim(a/A)=\lgdim(X)$.

\end{defn}

 An immediate corollary of the definition and the above observation is:
\begin{cor}\vlabel{union} If $X=\bigcup_{i=1}^n X_i$ is a finite union of definable sets then
$\lgdim(X)=\max_i \, \lgdim X_i$.
\end{cor}
\begin{fact} A definable $X\sub M^n$ is gp-short if and only if $\lgdim(X)=0$.
\end{fact}
\proof Without loss of generality $X$ is definably connected, defined over
$\emptyset$. If $X $ is gp-short then its projection on each coordinate is gp-short
so every tuple in $X$ is contained in $\shcl(\emptyset)$. Conversely, if some
projection of $X$ is gp-long then, by Fact \ref{shcl1}(1), this projection contains an
element of long dimension $1$, so $X$ contains a tuple of positive long dimension
over $\emptyset$.\qed

\begin{defn}  A $k$-long box is a cartesian product of
$k$ gp-long open intervals.

If $B=\Pi_{i=1}^n (c_i,d_i)$ is an $n$-long box in $M^n$, we say that $\bar
a=(a_1,\ldots, a_n)\in B$ is {\em long-central in $B$} if for every $i=1,\ldots, n$,
$a_i$ is long-central in $(c_i,d_i)$.

\end{defn}

Clearly, if $B$ is an $n$-long box defined over $A$, $a\in B$ and $\lgdim(a/A)=n$
then $a$ is long-central in $B$.

 The following is easy
to verify:
\begin{fact}\vlabel{short-in-long}Let $B\sub M^n$ be an $n$-long box and let $a$ be
long-central in $B$. If $C\sub B$ is some $A$-definable, definably connected,
gp-short set containing $a$, then the topological closure of $C$ in $M^n$ is
contained in $B$.
\end{fact}

%We can also probably prove the following (which we might need at some point)
%\\

%\noindent{\bf Conjecture} \emph{Let $B\sub M^n$ be an $n$-long box and let $a\in B$
%be long-central. Assume that $C\sub B$ is some $A$-definable, definably connected
%set containing $a$ and $\lgdim(C)<n$. Then there is an $A$-definable set $C'\sub
%\Cl(C)$, such that $\dim(C')<\dim(C)$.
%\\}

\begin{fact}\vlabel{nbox}
Assume that $X\sub M^n$ is an $A$-definable set, $a\in X$ and $\lgdim(a/A)=n$. Then
there exists $A_1\supseteq A$ and an $A_1$-definable $n$-long box $B$, such that
$a\in B$, $\Cl(B)\sub X$ and $\lgdim(a/A_1)=n$. In particular, $X\sub M^n$ has long
dimension $n$ if and only if it contains an $n$-long box.
\end{fact}
\proof We  use induction on $n$.

For $n=1$, if $X\sub M$ is $A$-definable then $a$ belongs to one of its definably
connected components, which is an $A$-definable interval containing $a$. Since
$\lgdim(a/A)=1$, $a$ must be long-central in it. We can then apply Fact \ref{shcl1}(3).

For $a\in M^{n+1}$, we may assume that $X$ is an $n+1$-cell and let $\pi:M^{n+1}\to
M^n$ be the projection onto the first $n$ coordinates.  We let $f,g:\pi(X)\to M$ be
the $A$-definable boundary functions of the cell $X$, with $f<g$ on $\pi(X)$.
Because $\lgdim(a)=n+1$, the interval $(f(\pi(a)), g(\pi(a)))$ is gp-long. Applying
\ref{shcl1}(3), we can find $e_1,e_2$, with $f(\pi(a))<e_1<a_{n+1}<e_2<g(\pi(a))$,
such that $a_{n+1}$ is long-central in $(e_1,e_2)$ and such that
$\lgdim(a/Ae_1e_2)=n+1$. Consider the first order formula over $Ae_1e_2$, in the
variables $x=(x_1,\ldots, x_n)$, which says that $f(x)<e_1<e_2<g(x)$. This is an
$Ae_1e_2$-definable property of $\pi(a)$, so by induction there exists an $n$-long
box $B\sub \pi(X)$, defined over $A_1\supseteq A$, and containing $\pi(a)$, with
$\lgdim(\pi(a)/A_1e_1e_2)=n$, such that for all $x\in B$, we have
$f(x)<e_1<e_2<g(x)$. The box $B\times (e_1,e_2)$ is the desired $n+1$-long box.\qed

%Assume may now that $a=(a',a'')\sub M^n$ with $a'\in M^k$ of long
%dimension $k$ over $A$, and $a''\sub \shcl(Aa')$. It follows that
%there are $A$-definable partial functions $f=(f_1,\ldots,
%f_{n-k})$ and $g=( g_1,\ldots, g_{n-k})$ from $M^k$ to $M^{n-k}$,
%such that each $a_i\in a''$ belongs to the gp-short interval
%$(f_j(a'),g_j(a')$. Let us assume that this gp-short interval is
%witnessed over $A_0\supseteq A$. We can therefore write a first
%order formula  (???? Problem

%By intersecting with $X$ we may assume that the fiber
%$\{a'\}\times \Pi_{j=1}^{n-k}(f_j(a'),g_j(a'))$  is contained in
%$X$.

% The domains of the $f$ and
%$g$'s are $A$-definable sets whose intersection we call $Y$. By restricting $Y$ if
%necessary (to another $A$-definable set), we may assume that for every $x'\in Y$,
%the fiber $\{x'\}\times \Pi_{j=1}^{n-k}(f_j(x'),g_j(x'))$ is contained in $X$.

%By the case we already handled, there is an $A_1$-definable box $B\sub M^k$
%containing $a'$ and contained in $Y\cap \pi(X)$ (with $\pi:M^n\to M^k$ the
%projection onto the first $k$-coordinates), such that $\lgdim(a/A_1)=\lgdim(a/A)=k$.

 We can now conclude:
\begin{fact}\vlabel{shcl-types} Assume that $\lgdim(a/A)=n$, for $a\in M^n$ and let
$p(x)=\tp(a/A)$. Then there exists an $n$-long box $B\sub M^n$, defined over
$A_1\supseteq A$, such that $a\in B\sub p(M)$, and $\lgdim(a/A_1)=n$.
\end{fact}
\proof  Write the type $p(x)$ as the collection of $A$-formulas $\{\phi_i(x):i\in
I\}$ and let $X_i=\phi_i(\CM)$.  We let $B(x,y)=\Pi_{j=1}^n(x_j,y_j)$ be a
variable-dependent $n$-box, and consider the type $q(x,y)$ which is the union:
\begin{equation*}
\{\Cl(B(x,y))\sub X_i: i\in I\}\cup \mbox{``$B(x,y)$ is an $n$-long box''} \cup
``\lgdim(a/xyA)=n''.
\end{equation*}
By Fact \ref{int-fact0}, $q(x,y)$ is indeed a type over $A$. By Fact \ref{nbox}, the type is
consistent, so we can find a box as needed.\qed

\begin{lemma}\vlabel{gp-cl2} Assume that $\{X_t:t\in T\}$ is a definable family of
subsets of $M^n$, with $X=\bigcup_{t\in T} X_t$. Assume that
 $\lgdim(T)\leq \ell$ and for every  $t\in T$, we have $\lgdim(X_t)\leq k$. Then
 $\lgdim(X)\leq k+\ell$.
\end{lemma}
\proof  Without loss of generality, $X$ is $\emptyset$-definable. Take $x\in X$ with
$\lgdim(x/\emptyset)=\lgdim(X)$, and choose $t\in T$ so that $x\in X_t$. We then have
$$\lgdim(xt/\emptyset)=\lgdim(x/t)+\lgdim(t)=\lgdim(t/x)+\lgdim(x).$$
 By our assumptions, $\lgdim(x/t)\leq k$ and $\lgdim(t)\leq \ell$, hence
 $\lgdim(t/x)+\lgdim(x)=\lgdim(xt)\leq k+\ell$. It follows that $\lgdim(x)\leq k+\ell$
 so $\lgdim(X)\leq k+\ell$.
\qed

\section{Functions on gp-long and gp-short intervals, and the main theorem}

\begin{lemma}\vlabel{simple}
1. Let $I$ be a gp-long interval, and assume that $f:I\to M$ is
$A$-definable, continuous and strictly monotone. Let $t_0$ be
long-central in $I$. For every $t\in M$, let
$$\Sh(t)=\{ x\in M:\mbox{ the distance between $x$ and $t$ is
gp-short} \}.$$ Then $\Sh(t_0)\sub I$ and $f(\Sh(t_0))=\Sh(f(t_0)).$

%Assume that $f:I\times J\to M$ is definable, continuous and strictly monotone in
%both variables (possibly increasing in one and decreasing in the other).

%2. If $(t_1, t_2)\sub I$ and $(x_1,x_2)\sub J$ are gp-short then the distance
%between $f(x,t_1)$ and $f(x_2,t_2))$ is gp-short.

%3. If $(t_1,t_2)\sub I$ is gp-short and $(x_1,x_2)\sub J$ is gp-long then the
%distance between $f(x_1,t_1)$ and $f(x_2,t_2))$ is gp-long.
\end{lemma}
\proof  It is clear that $\Sh(t_0)\sub I$. Because $f$ is continuous it sends
elements whose distance is gp-short to elements of gp-short distance, namely
$f(\Sh(t_0))\sub \Sh(f(t_0))$. Because $f$ is strictly monotone, $J=f(I)$ is also
gp-long and $f(t_0)$ is long-central in $J$. We now apply the same reasoning to
$f^{-1}|J$ and conclude that $f(\Sh(t_0))=\Sh(f(t_0))$.

%(2) follows from the fact that both points are contained in the image of
%$[t_1,t_2]\times [x_1,x_2]$ under $f$.

%(3) First note that  the distance between $f(x_1,t_1)$ and
%$f(x_1,t_2)$  is gp-short. Then, note that the distance between
%$f(x_1,t_2)$ and $f(x_2,t_2)$ is gp-long. It follows that the
%distance between
% $f(x_1,t_1)$ and $ f(x_2,t_2)$  is gp-long.\qed

\begin{lemma}\vlabel{loc-const} Assume that $f:X\to M$ is an  $A$-definable function with $\lgdim(X)=k>0$.
If $f(X)$ is gp-short then there are finitely many $y_1,\ldots, y_m\in M$, all in
$\dcl(A)$, such that $\lgdim(X\setminus f^{-1}(\{y_1,\ldots, y_m\})<k$. In
particular, $f$ is locally constant at every long-generic point in $X$.
\end{lemma}
\proof The set of all points in $X$  at which $f$ is locally constant is definable
over $A$ and has finite image. It is therefore sufficient to prove that $f$ is
locally constant at every $a\in X$, with $\lgdim(a/A)=k$.

If  $b=f(a)$ then $k=\lgdim(ab/A)=\lgdim(a/Ab)+\lgdim(b/A)$. But $b\in f(X)$, a
gp-short set defined over $A$ and therefore $\lgdim(b/A)=0$. It follows that
$\lgdim(a/Ab)=k$, so in particular, $\dim(a/Ab)=k$. It follows that there is a
neighborhood of $a$, in the sense of $X$, on which $f(x)=b$.\qed

As a corollary we have:

 \begin{lemma}\vlabel{loc-const2} Assume that $X\sub M^{k+1}$ is
definable, $\dim(X) =k+1$, $\lgdim(X)=k$ and the projection  $\pi(X)$ onto the last
coordinate is gp-short. Then $X$ contains a definable set of the form $B\times J$,
for $B\sub X$ a $k$-long box and $J\sub M$ an open gp-short interval.
\end{lemma}
\proof Take $\la a,b\ra$ generic in $X$, with $\lgdim(a)=k$. Since $\dim X=k+1$ and
$\la a,b\ra $ is generic in $X$, there exists an interval
$J=(\sigma_1(a),\sigma_2(a))$, for some $\emptyset$-definable functions $\si_1,\si_2$,  such
that $\{a\}\times J\sub X$. The functions $\sigma_1,\si_2$ take values in the closure
of $\pi(X)$, namely in a gp-short set. By Lemma \ref{loc-const}, the functions are
locally constant on a $0$- definable set $Y\sub M^k$ containing $a$, so we can
finish by Lemma \ref{nbox}.\qed

Here is our main lemma:

\begin{lemma}\vlabel{main-lemma} Assume that $L\sub M^n$ is definable, $\lgdim L=k$, $J\sub M$ a gp-short open
interval and $F:L\times J\to M$ definable over $A$. If $F(L\times J)$ is gp-short
then there exist an $A$-definable set $S\sub L$ with $\lgdim(S)<k$ and
finitely many $A$-definable partial functions $g_1,\ldots,g_K:J\to M$ such that for
all $\ell\in L\setminus S$, and all $x\in J$, we have
$$\bigvee_{i=1}^K f(\ell,x)=g_i(x).$$
\end{lemma}
 \proof For every $x\in J$, let $f^x:L\to M$ be defined by $f^x(\ell)=f(\ell,x)$.
By Lemma \ref{loc-const}, there exists an $Ax$-definable set $L_x\sub X$ such that
$f^x$ is locally constant on $L_x$ (in particular, $f^x(L_x)$ is finite) and
$\lgdim(X\setminus L_x)<k$. By o-minimality, there exists a uniform bound $K$ on the
size of $f^x(L_x)$, so we can  define (possibly partial) functions $g_i:J\to M$,
$i=1,\ldots, K$, such that for every $\ell\in L_x$, we have $f(\ell,x)=g_i(x)$ for
some $i=1,\ldots, K$.
 If we let  $S=\bigcup_{x\in
J}(X\setminus L_x)$ then,  by Lemma \ref{gp-cl2}, we have $\lgdim(S)<k$. \qed

As a corollary we have:
\begin{lemma}\vlabel{main-lemma2} Let $L\sub M^2$ be a definable set, with $\lgdim(L)=2$,
 and $J=(a,b)$ a gp-short interval. Assume that $f:L\times J\to M$ is a
definable function and that $a_0$ is generic in $J$ (in the usual sense).

Then there exist a $2$-long box $B\sub L$, an open interval $J'\sub J$ containing
$a_0$, and a definable partial, two-variable function $g:M^2\to M$, such that for
every $\ell\in B$ and $x\in J'$, we have
$$f(\ell,x)=g(f(\ell,a_0),x).$$
\end{lemma}
\proof  Assume that all data is definable over $\emptyset$. By Lemma
\ref{loc-const2}, we may assume that $f$ is continuous (we apply the lemma to the
set of all points in $L\times J$ at which $f$ is continuous). Fix $a_0\in J$, and
for every $y\in M$ let
$$L^y=\{\ell\in L:f(\ell,a_0)=y\}.$$ Notice that for every $\ell\in L^y$, the image
$J^y=f(\{\ell\}\times J)$ is a gp-short interval containing $y$. It follows from
Theorem \ref{gp-short} that the union $\bigcup_{\ell\in L^y} J^y=f(L^y\times J)$ is
gp-short. We can therefore apply Lemma \ref{main-lemma} to $f|L^y\times J$. Hence,
there exists a number $k_y$, definable functions $g_{1,y}(x), \ldots, g_{k_y,y}(x)$
and a definable set $S^y\sub L^y$ with $\lgdim S^y<\lgdim L^y$,  such that for all
$\ell \in L^y\setminus S^y$ and all $x\in J$ we have
$$\bigvee_{i=1}^{k_y}f(\ell,x)=g_{i,y}(x)=g_{i,f(\ell,a_0)}(x).$$

By o-minimality, the number $k_y$ is bounded uniformly in $y$, so we can find $K$,
and  definable partial two-variable functions $g_{i}(x,y)$, $i=1,\ldots, K$, such
that for every $y\in M$, $\ell\in L^y\setminus S^y$ and $x\in J$,
$$\bigvee_{i=1}^{K}f(\ell,x)=g_{i}(y,x)=g_{i}(f(\ell,a_0),x).$$

If we let  $S=\bigcup_{y}S^y$ then, by the dimension formula (similar to
Lemma \ref{gp-cl2}) $\lgdim (S)<\lgdim(\bigcup_y L^y)=\lgdim L$ and therefore $L\setminus
S$ contains a $2$-long box $B$. Finally, for $i=1,\ldots, K$, let $X_i$ be all
$(\ell,x)\in B\times J$ such that $f(\ell,x)=g_i(f(\ell,a_0),x)$. Each $X_i$ is
$\emptyset$-definable, $B\times J=\bigcup_i X_i$, so at least one of these $X_i$'s
contains a point $\la \ell,a_0\ra$ with $\lgdim(\ell/a_0)=2$. It follows that this
$X_i$ contains a box of the form $B'\times J'$ with $\lgdim(B')=2$ and $J'\ni a_0$
an open interval. We now have, for every $\ell\in B'$ and $x\in J'$,
$$f(\ell,x)=g_i(x,f(\ell,a_0)).$$\qed

Until now we did not use at all the Trichotomy Theorem for o-minimal structures. The
next result requires it.

\begin{cor} \vlabel{main3} Assume that $I_1,I_2,I_3$ are gp-long intervals.
Let $f:I_1\times I_2\times I_3\to M$ be a definable  function. Then there are
gp-long intervals $I_1'\sub I_1, I_2'\sub I_2$ and $I_3'\sub I_3$ such that for all
$\la a_1,b_1\ra,\la a_2,b_2\ra \in I_1'\times I_2'$,
\begin{equation}\label{eqq1} \begin{array}{c} \exists x\in I_3'\,( f(a_1,b_1,x)=f(a_2,b_2,x) )\\  \Leftrightarrow \\ \forall
x\in I_3' \,(f(a_1,b_1,x)=f(a_2,b_2,x))\end{array}\end{equation}

Namely, the family of functions $f(a,b,-)|I_3'$, for $\la a,b\ra\in I_1'\times
I_2'$, is at most 1-dimensional.
\end{cor}
\proof Without loss of generality, $f$ is continuous. We assume that all data are
definable over $\emptyset$. Fix $\la a_1,a_2,a_0\ra \in I_1\times I_2\times I_3$
with $\lgdim(a_1,a_2,a_0/\emptyset)=3$.

Assume first that there is no gp-short open interval containing $a_0$. In this case, by the Trichotomy Theorem, $a_0$ is a trivial point, so the function
$f(a_1,a_2,-)$ is either constant around $a_0$, namely equals  some $g(a_1,a_2)$,
or equals  some $\emptyset$-definable $1$-variable function $h(-)$. In either case there
are gp-long $I_j'\sub I_j$, $j=1,2,3$, for which we either have
$f(a_1',a_2',x)=g(a_1',a_2')$ or $f(a_1',a_2',x)=h(x)$, for all $(a_1',a_2',x)\in
I_1'\times I_2'\times I_3'$. In either case (\ref{eqq1}) holds.

We can therefore assume that there is some gp-short interval $J$ around $a_0$.  By
Lemma \ref{main-lemma2}, there are gp-long intervals $I_1'\sub I_1$ and $I_2'\sub
I_2$  such that
\\

\noindent ($\star$) {\em for every $\ell_1,\ell_2\in I_1'\times I_2'$,
the functions $f(\ell_1,-)=f(\ell_2,-)$ agree in some neighborhood of $a_0$ if and
only if $f(\ell_1,a_0)=f(\ell_2,a_0)$}.
\\

By Fact \ref{shcl1}(3),  we can choose the intervals $I_1'=(a',b')$ and
$I_2'=(a'',b'')$ so that $a_0\notin \shcl(a'b'a''b'')$.
 Because $a_0$ is $\shcl$-generic in $I_3$, and (*) is a first order formula about $a_0$ over $a'b'a''b''$, there is
 a gp-long interval $I_3'\sub I_3$ containing $a_0$ such that for all $x\in I_3'$ we have

 If $\ell_1,\ell_2\in I_1'\times I_2'$, then
 $f(\ell_1,-)$ and $f(\ell_2,-)$ agree on a
neighborhood of $x$ if and only if $f(\ell_1,x)=f(\ell_2,x)$.

But now, by continuity and definable connectedness of $I_3'$ if $f(\ell_1,-)$ and
$f(\ell_2,-)$ agree anywhere in $I_3'$ then they must agree everywhere on
$I_3'$.\qed

We now reach our main theorem of this section:
\begin{theorem}\vlabel{mainthm} Let $f:I\times J\to M$ be a definable  function which is
strictly monotone in each  variable separately. Then either $I$ or $J$ is gp-short.
\end{theorem}
\proof We start by assuming, for contradiction, that both $I$ and $J$ are gp-long.
Write $I=(a,b)$ and $J=(c,d)$. The general idea is  that outside of subsets of
$I\times J$ of long dimension smaller than $2$, we have a phenomenon similar to local
modularity (every definable family of curves is one-dimensional) and therefore we
can apply the standard machinery of local modularity to produce a definable group.

 For $x\in I$, we write $f_x(y):=f(x,y)$. By partitioning $I\times J$ into finitely many sets,
 and by applying Fact \ref{nbox}, we may assume that $f$ is continuous and for every $x\in I$,
 $f_x$ is strictly monotone, say increasing.

\begin{claim}\vlabel{claim1}
There exists a gp-long interval $K$ and a gp-long interval $I_1 \sub I$ such that
for all $x\in I_1$, we have $K\sub f_x(J)$.
\end{claim}
\proof  Take $x_0\in I$ to be $\shcl$-generic. The interval $f(x_0,J)$ is gp-long so
we can find $y_0$ in it which is $\shcl$-generic over $x_0$, and so
$\lgdim(x_0,y_0/\emptyset)=2$. The set $\{(x,y)\in I\times M:y\in f(x,J)\}$ is
$\emptyset$-definable and contains $\la x_0,y_0\ra$, hence there is a cartesian
product $I_1\times K$ of two gp-long intervals which is contained in it.\qed

%For every $x\in I$, let $f_1(x)=f_x(c)$ and $f_2(x)=f_x(d)$. By assumption,
%$f_1(x)<f_2(x)$ and furthermore, since $J$ is gp-long, the interval $(f_1(x),f_2(x)$
%is gp-long. By moving to a subinterval of $I$, if needed, we may assume that $f_1$
%and $f_2$ are gp-long. But then we have a cell $C=\{x,y):x\in I\,\&\,
%f_1(x)<y<f_2(x)\}$, so by Fact \ref{int-fact2}, there exists a gp-long interval
%$I_1\sub I$ and a gp-long interval $K$ such that $I_1\times K\sub C$. This implies
%that $K\sub (f_1(x),f_2(x))=f_x(J)$ for all $x\in I_1$,as needed.\qed

To simplify notation, we assume that for all $x\in I$, we have $K\sub f_x(J)$.  We
can now consider the family of functions $\{f_xf_y^{-1}|K:x,y\in I\}$ as a
collection of continuous functions from $K$ into $M$. Let
$$F(x,y,t)=f_xf_y^{-1}(t).$$ The function $F$ is a map from $I\times I\times K$. We apply Corollary
\ref{main3}, and find $I_1\sub I$, $I_2\sub I$, and $I_3\sub K$ all gp-long such
that for every $x,x'\in I_1$, $y,y'\in I_2$ and $t\in I_3$, if $F(x,y,t)=F(x',y',t)$
then for all $t'\in K$ we have $F(x,y,t')=F(x',y',t')$. Namely, for all $x_1,x_1'\in
I_1$ and $x_2,x_2'\in I_2$,
\begin{equation}\label{eqq2}
\exists t\in I_3\,\, f_{x_1}f_{x_2}^{-1}(t)=f_{x_1'}f_{x_2'}^{-1}(t) \Leftrightarrow
\forall t\in I_3 \,\,f_{x_1}f_{x_2}^{-1}(t)=f_{x_1'}f_{x_2'}^{-1}(t)\end{equation}

We  now fix $\la x_0,y_0,t_0\ra$ long-generic in $I_1\times I_2\times I_3$ and let
$w_0=f_{x_0}f_{y_0}^{-1}(t_0)$. We also let $a_0=f_{y_0}^{-1}(t_0)$. By Lemma
\ref{simple}, if $t\in \Sh(t_0)$ then we have $f^{-1}_{y_0}(t)\in \Sh(a_0)$, hence by the same lemma,  the map $y\mapsto f_{y}^{-1}(t)$ sends $\Sh(y_0)$ bijectively onto
$\Sh(a_0)$. Similarly, for every $y\in \Sh(y_0)$, the map $t\mapsto f_{y}^{-1}(t)$
sends $\Sh(t_0)$ bijectively onto $\Sh(a_0)$ and  for every $x\in \Sh(x_0)$, the map
$a\mapsto f_x(a)$ sends $\Sh(a_0)$ bijectively onto $\Sh(w_0)$. Thus, if $x_1,x_2\in
\Sh(x_0)$ then the function $f_{x_1}^{-1}f_{x_2}$ is a permutation of $\Sh(a_0)$.

\begin{claim}\vlabel{group-claim}
(1) For every $x_1,x_2\in \Sh(x_0)$ there is a unique $x_3\in \Sh(x_0)$
such that $f_{x_2}^{-1}f_{x_3}=f_{x_0}^{-1}f_{x_1}$, as functions from $\Sh(a_0)$ to
$\Sh(a_0)$.

(2) For every $x_1,x_3 \in \Sh(x_0)$ there exists a unique $x_2 \in \Sh(x_0)$ such
that $f_{x_2}^{-1}f_{x_3}=f_{x_0}^{-1}f_{x_1}$.
\end{claim}
\proof We prove (1) -- the proof of (2) is similar. Consider first $f_{x_1}f_{y_0}^{-1}(t_0)\in \Sh(w_0)$. By the above
observations, there exists a unique $y_1\in \Sh(y_0)$ such that
$$f_{x_1}f_{y_0}^{-1}(t_0)=f_{x_0}f_{y_1}^{-1}(t_0).$$
By the same reasoning, there exists a unique $x_3\in \Sh(x_0)$, such that
$$f_{x_3}f^{-1}_{y_0}(t_0)=f_{x_2}f_{y_1}^{-1}(t_0).$$
By (\ref{eqq2}), the above two equalities at the point $t_0$ translate to equality
of functions on $\Sh(t_0)$. Using composition and substitution we obtain
$$f_{x_2}^{-1}f_{x_3}=f_{x_0}^{-1}f_{x_1}.$$\qed

We are now ready to define a group-operation on $\Sh(x_0)$ with identity $x_0$: For $x_1,x_2, x_3\in
\Sh(x_0)$, $$ x_1+x_2=x_3 \Leftrightarrow f_{x_2}^{-1}f_{x_3}=f_{x_0}^{-1}f_{x_1},
\mbox{ as functions on $\Sh(a_0)$}.$$
Claim \ref{group-claim}(1) implies that $+$ is associative, Claim
\ref{group-claim}(2) guarantees an inverse (with $x_3=x_0)$, and commutativity
follows from one-dimensionality of $\Sh(x_0)$ and o-minimality.

Although $\Sh(x_0)$ is not a definable set (it is $\bigvee$-definable)  the same
operation can be defined on $I_1'\supseteq \Sh(x_0)$ (but might only be partial
there) by: $\la x_1,x_2,x_3\ra\in R$ if and only if
$f_{x_2}^{-1}f_{x_3}=f_{x_0}^{-1}f_{x_1}$ agree on some neighborhood of $a_0$. This
is a definable relation which when restricted to $\Sh(x_0)^3$ gives a group
operation. Using compactness, one can show that that the restriction of the
operation to some gp-long interval $I_1''$ containing $\Sh(x_0)$ yields a
generalized group-interval. This contradicts the definition of a gp-long interval,
so returning to our original assumptions, either $I$ or $J$ must be gp-short.\qed

The following is a generalization we will require later:
\begin{cor} \vlabel{injective} Let $f:I_1\times \dots\times I_{n+1}\to M^n$ be a
definable function which is injective in each variable separately
(namely for every $a_1,\ldots, a_{i-1},a_{i+1},\ldots,a_n$ in
$I_1,\ldots, I_{i-1},I_{i+1}, \ldots, I_n$, respectively, the
map $f(a_1,\ldots, a_{i-1},x,a_{i+1},\ldots, a_n):I_i\to M^n$ is
injective).

Then at least one of the intervals $I_j$ is gp-short.
\end{cor}
\proof Assume towards contradiction that all intervals are
gp-long.

\begin{lemma} If $h:I_1\times \cdots \times I_{m}\to M$ is a definable
function on a product of gp-long intervals, then there exist
gp-long subintervals $J_i\sub I_i$, $i=1,\ldots, m$, there exists
$j\in \{1,\ldots, m\}$, and a definable $g:J_j\to M$, such that
for every $a=(a_1,\ldots, a_{m})\in J_1\times \cdots \times
J_{m}$, we have $h(a)=g(a_j)$.
\end{lemma}
\proof We take $a=(a_1,\ldots, a_{m})\in I_1\times\cdots\times I_{m}$ which has long
dimension $m$ (over the parameters defining everything). Then there is a long
interval $J_{m}\sub I_{m}$ containing $a_{m}$ and defined over parameters $A$ so
that the long dimension of $a$ over $A$ is still $m$, and such that $h(a_1,\ldots,
a_{m-1},x)$ is either constant (namely of the form $g(a_1,\ldots, a_{m-1})$) or
strictly monotone. In the first case we can use induction on $m$ to finish the
argument. In the second case we prove that $h(a_1,\ldots, a_{m-1})$ is a function of
the last coordinate only: We consider the $m-1$-st coordinate and, using the same
reasoning as before, we find a gp-long interval $J_{m-1}\sub I_{m-1}$ on which the
map $f(a_1,a_2,\ldots, x,a_m)$ is either constant or strictly monotone. But now,
this second possibility is impossible, or else the map $f(a_1,\ldots, a_{m-2},x,y)$
is strictly monotone in both variables on $J_{m-1}\times J_m$, implying, using
Theorem \ref{mainthm}, that one of these intervals is gp-short, contradiction. We
are then left with the case where $f(a_1,\ldots, ,x,a_m)=g(a_1,a_2,\ldots, a_{m-2},
a_m)$ for $x\in J_{m-1}$ and proceed in the same manner.\qed

We now return to our proof of the corollary and to the assumption that all intervals
are gp-long. We write $f(a)=(f_1(a),\ldots, f_n(a)).$ If we apply the last lemma to
$f_1$ then we can find gp-long intervals $J_i\sub I_i$, $i=1,\ldots, n+1$, and a
definable one-variable function $h(x)$ such that
 for every $(x_1\ldots, x_{n+1})\in J_1\times \cdots \times J_{n+1}$, we have, without loss of
generality,
 $f_1(x_1,\ldots, x_{n+1})=h(x_1)$. But now, fix any $a_1\in
J_1$ and consider the function $$F(x_2,\ldots,
x_{n+1})=(f_2(a_1,x_2\ldots,x_{n+1}),\ldots, f_n(a_1,x_2,\ldots,x_{n+1}))$$ from
$J_2\times \cdots \times J_{n+1}$ into $M^{n-1}$. It is easy to see that the
function $F$ is still injective in each coordinate so by induction, one of the
$J_i$, $i=2,\ldots, n+1$, is gp-short, contradiction. We then conclude that one of
the $I_j$'s must be gp-short.\qed

\qed

Here is a first observation about definable groups.

\begin{cor}\vlabel{group-1} Let $G\sub M^n$ be a definable group in
an o-minimal structure. Then every definable $1$-dimensional subset of $G$ is
group-short.
\end{cor}
\proof Let $J\sub G$ be a definable $1$-dimensional set, identified with a finite
union of intervals in $M$, and apply Corollary \ref{injective} to the map
$f:J^{n+1}\to G$ given by the group product
$$f(g_1,\ldots, g_{n+1})=g_1g_2\ldots g_{n+1}.$$ It follows that $J$ must be gp-short.\qed

\section{$\CM$-quotients}

\subsection{Dimension of elements in $\CM^{eq}$}

\begin{defn} Let $X_1,\ldots, X_n$ be pairwise disjoint definable subsets of
$M^{k_1}, \ldots, M^{k_n}$, respectively, and let $X=X_1\sqcup\cdots \sqcup X_n$. A
subset $W\sub X$ is called {\em definable} if $W\cap X_i$ is definable for every
$i=1,\ldots, n$. For $\dim W$ we take the maximum of all $\dim W\cap X_i$.

 Note that $X^k$ can be similarly written as a finite pairwise
disjoint union of cartesian products of the $X_i$'s, and a subset of $X^k$ is called
{\em definable} accordingly. If $E$ is a definable equivalence relation then we say
that $X/E$ is an {\em $\CM$-quotient}.

 A subset of $X/E$ is called {\em definable} (we should
say ``interpretable'' but it sounds awkward) if it is the image of a definable
subset of $X$ under the quotient  map.
\end{defn}

Note that $X$ above is in definable bijection with an actual definable subset of
some $M^k$, after naming parameters, but it is often more natural to consider it as
as the union of definable sets in various $M^k$'s.

For $A\sub \CM^{eq}$ a small set of parameters and $a\sub M$, the closure operation
$\dcl(aA)$ still defines a pre-geometry on $M$ so $\dim(a/A)$ makes sense.
\begin{defn} Let $A\sub \CM^{eq}$ be a small set, $X\sub M^k$ an $A$-definable set and $E$ an
$A$-definable equivalence relation on
$X$.

For $g\in X/E$, we define $\dim(g/A)$ to be the maximum among $\dim(x/A)-\dim[x]$,
as $x$ varies in the class $g$. For $Y\sub X/E$ definable over $A$, we let
$$\dim Y=\max\{\dim(g/A):g\in Y\}.$$

If for $g\in Y$ we have $\dim (Y)=\dim(g/A)$ then $g$ is called a {\em generic
element of $Y$ over $A$}.
\end{defn}

One can show that the above definition does not depend on $A$, namely if we
calculate the dimension of $Y$ with respect to a larger set of parameters
$B\supseteq A$ then we obtain the same result. Here are some more basic properties.

\begin{fact}\label{int-dim-props}\begin{enumerate}

 \item For $g,h\in X/E$ and $A\sub \CM^{eq}$, we  have
$$\dim(g,h/A)=\dim(g/hA)+\dim(h/A).$$

\item Assume that $g=[a]$ for $a\in X$, and $\dim(a/A)=k$. Then $\dim(g/A)\leq k$.
\end{enumerate}
\end{fact}
\proof (1) is in \cite[Proposition 3.4]{ga}.

 (2). Since $\dim(g/aA)=0$, the dimension formula implies that $\dim(a/gA)+\dim (g/A)=\dim(a/A)=k$ and hence
 $\dim(g/A)\leq k$.\qed

The following is a direct corollary of the dimension formula.

\begin{claim}\vlabel{dimeqrel}
Let $T\subset M^{eq}$ be a definable set, and let $\{X_t\scht t\in T\}$ be a
definable family of pairwise disjoint definable sets in $\CM^{eq}$. If the
dimension
of each $X_t$ is $r$ and $\dim T=e$ then $\dim \bigcup_{t\in T} X_t=r+e$.
\end{claim}

Recall:

\begin{defn} For $X, Y$ definable sets and $E_1,E_2$ definable equivalence relations
on $X$ and $Y$, respectively, a function $f:X/E_1\to Y/E_2$ is called {\em
definable} if the set $\{\la x,y\ra\in X\times Y:f([x])=([y])\}$  is definable.

\end{defn}
We will need the following general fact about definable equivalence relations:

\begin{claim}\vlabel{general-eq}
Let $X\sub M^k$ be an $A$-definable set and $E$ an  $A$-definable equivalence
relation on $X$. Then there exists an $\CM$-quotient $Y/E'$, defined over $A$,  and
an $A$-definable bijection $f:X/E\to Y/E'$ such that $Y$ can be partitioned into
finitely many definable sets $U_1,\ldots, U_m$ with the following properties:
\begin{enumerate}
\item Each $U_i$ is an open subset of $M^{k_i}$.

\item Each $E'$-class is contained in a single $U_i$.

\item For each $i=1,\ldots,m$, there exists $d_i\in \mathbb N$ such that every
 $E'$-class in $U_i$ is a set of dimension $d_i$, and the projection
$\pi_{d_i}:M^{k_i}\to M^{d_i}$ onto the first $d_i$ coordinates is a homeomorphism.
\end{enumerate}
\end{claim}
\begin{proof} We prove the result by induction on $n=\dim X$.

First, partition $X$ into a finite union of sets, in each of which every
$E$-class
has the same dimension, and such that every $E$-class is contained in exactly
one of
these sets. Clearly, we can prove the result separately for each of these sets.
Thus, it is enough to prove the claim under the assumption that all classes have
the
same dimension $d$.

Let $\mathcal C$ be a cell decomposition of $X$. Each cell $C_i\in\mathcal C$
homeomorphically projects onto an open subset of $M^{n_i}$ for some $n_i\le n$. Let
$Z_1,\ldots,Z_r$ be these projections. We now consider the restriction of $E$ to
$Z_1$ and prove the claim for $Z_1$ (note that after doing that we plan to discard
all elements in $X\setminus Z_1$ whose classes intersect $Z_1$).

If $\dim(Z_1)<n$, then the claim holds for it by induction. Otherwise, $\dim
Z_1=n$, and thus necessarily, $\dim(Z_1/E)=n-d$.
 Let
$$Z_1'=\{x\in Z_1:\dim([x]\cap Z_1)<d\}.$$
Since $\dim(Z_1'/E)\leq n-d$  and every $[x]\cap Z_1'$ has dimension smaller than
$d$, it follows from Claim \ref{dimeqrel} that $\dim Z_1'<n$.  Moreover, for every
$x$, either $[x]\cap Z_1\subset Z_1'$ or $[x]\cap Z_1'=\emptyset$, so proving the
claim for $Z_1'$ and $\tilde Z_1=Z_1\setminus Z_1'$ is sufficient. By induction, the
claim holds for $Z_1'$.

By Lemma \ref{uniform}, we can uniformly partition all the equivalence classes
in
$\tilde Z_1$ into cells, then  choose a $d$-dimensional cell from each
equivalence
class in $\tilde Z_1$, and replace $\tilde Z_1$ by the union of these cells
(still
calling it $\tilde Z_1$). Note that omitting the remaining part of each class
does
not change the quotient. Next, we partition $\tilde Z_1$ into finitely many
sets, so
that in a single set, the cell of each class is of the same type (by that we
mean
that the projection onto the same $d$ coordinates is a homeomorphism). Since the
partition respects the classes, we may deal with each part separately.

Any set in this partition with dimension less than $n$ is handled by induction, so
we may only consider the sets of dimension $n$. We assume  then that $\tilde Z_1$ is
an $n$-dimensional union of $d$-dimensional cells, all of the same type. By
permutation of variables, we can suppose that projection on the first
$d$-coordinates is a homeomorphism of each class onto an open subset of $M^d$.

Now let $\mathcal D$ be a cell decomposition of $\tilde Z_1$, and let $B$ be the union
$$\bigcup_{D\in\mathcal D,\, \dim D<n} D.$$
Because $\tilde Z_1\sub M^n$ and $\dim \tilde Z_1=n$, the union of all
$n$-dimensional cells in $\mathcal D$ is an open subset of $M^n$, so $\tilde
Z_1\setminus B$ is still open in $M^n$. Thus, for each $x\in \tilde Z_1$, if the set
$[x]\cap (\tilde Z_1\setminus B)$ has dimension smaller than $d$ then it must be
empty (here we use the fact that $[x]\sub \tilde Z_1$ is a $d$-cell). Hence, a
class $[x]$ which intersects $\tilde Z_1\setminus B$ might not be a cell
anymore, but it is still true that its projection onto the first $d$ coordinates is a
homeomorphism onto an open subset of $M^d$. Hence $\tilde Z_1\setminus B$ satisfies
the claim. We now remove from $B$ all classes which are already represented in
$\tilde Z_1\setminus B$ (we still call the new set $B$) and handle $B/E$ by
induction on dimension. We therefore showed that the claim holds for $\tilde Z_1$
and hence also for $Z_1$.

Note that
the above argument only used the fact that $Z_1$ was an $n$-dimensional subset
of
$M^n$ (and that every class in $X$ has dimension $d$).

Next, remove all classes from $Z_2,\ldots,Z_r$ with representatives in $Z_1$. We
still use $Z_2,\ldots, Z_r$ for the remaining sets. Clearly, each class which is
contained in the new  $Z_2\cup \cdots \cup Z_n$ still has dimension $d$. If
$\dim
Z_2=n$ then we handle it exactly as we handled $Z_1$, and if $\dim Z_2<n$ then
we
apply induction. We proceed in the same manner until handle all $Z_i$'s and thus
prove the claim for $X$.\end{proof}

\subsection{Elimination of one dimensional quotients}

 Let $\{X_t:t\in T\}$ be a definable family of sets.
We say that $f:T\to M^n$ is an $\CF$-map if for every $t,s\in T$, if $X_t=X_s$ then
$f(t)=f(s)$. We say that $f$ is $\CF$-injective if in addition, whenever $f(t)=f(s)$
we have $X_t=X_s$.

We will use the following fact \cite[Claim 1.1]{HPP}:
\\

\noindent {\bf DEQ}: \emph{If $E$ is a $\emptyset$-definable equivalence relation on $X$
with finitely many classes then every class is $\emptyset$-definable.}
\\

\begin{theorem}\vlabel{onedim-equiv} Let $\CF=\{X_t:t\in T\}$ be
 a definable family of definable sets in $M^k$, with $T\sub \CM^{eq}$ and $\dim
 T=1$. Then there exists a definable $\CF$-injective map $f:T\to M^m$, for some $m$,
possibly over parameters.
\end{theorem}
\proof
Note that if  $T_1\sqcup T_2=T$ is a partition of $T$ and we let $\CF_1=\{X_t:t\in
T_1\}$ and $\CF_2=\{X_t:t\in T_2\}$ be the corresponding families then it is enough
to obtain $\CF_i$-injective functions for each $i=1,2$. As well, note that
if $\CF_1=\{Y_t\scht t\in T\}$ and $\CF_2=\{Z_t\scht t\in T\}$ are definable
families such that $\langle Y_t,Z_t\rangle = \langle Y_s,Z_s\rangle\iff
X_t=X_s$, then it is enough to obtain $\CF_i$-injective functions for each $i=1,2$.

Let $X=\bigcup_t X_t$. We go by induction on $k$.

If $\dim(X)<k$, then take a cell decomposition $C_1,\ldots,C_m$ of
$X$. Intersecting $X_t$ with each $C_1,\ldots,C_m$ yields a finite set of
families indexed by $T$, $\CF_i=\{X_t\cap C_i\scht t\in T\}$, for
$i=1,\ldots,m$. After a finite partition of $T$ based on whether $X_t\cap C_i=\emptyset$ for $i=1,\ldots,m$, it is then enough to find $\CF_i$-injective
functions, which we have by induction, since each $C_i$ is in definable
bijection with a subset of $M^{k-1}$. Thus, we may suppose that $\dim(X)=k$.

By replacing $T$ with $T/\sim$, with $t\sim s$ if and only if $X_t=X_s$, we may
assume that $X_t=X_s$ if and only if $t=s$, and still $\dim T=1$ (if $\dim T=0$ then
we are done by DEQ).

By Lemma \ref{uniform}, we can uniformly partition each set $X_t$ into a disjoint
union of cells $X_t^1,\ldots, X_t^m$ (in particular, the partition depends only on
the set $X_t$ and not on $t$), and let $\CF_i=\{X_t^i:t\in T\}$, $i=1,\ldots, m$.
Then it is sufficient to define $\CF_i$-injective maps. Indeed, if we have such
$f_i:T\to M^n$ (without loss of generality we can assume that they all go into the
same $n$) then we may now define $h:T\to (M^n)^m$ by $h(t)=(f_1(t),\ldots, f_m(t))$.

By a further partition of $T$,  we may assume that all $X_t$'s are cells in
$M^k$ of the same dimension $r$. We may suppose that we still have $\dim(X)=k$,
since otherwise our above argument works to finish $X$ by induction. We prove
the result by induction on both $r$ and $k$.

\bigbreak
 \noindent{\bf Case 1.} $r=k$.

In this case, each $X_t$ is bounded above and below by two $k-1$ cells $Y_t^1,
Y_t^2$, of dimension $k-1$, which determine the set $X_t$. By considering the
families $\{Y^1_t:t\in T\}$ and $\{Y^2_t :t\in T\}$ and applying induction we will
be  done.

\bigbreak
\noindent{\bf Case 2.} $r<k$.

\vspace{.3cm}

Let $X^0$ be the collection of all  $x\in X$ which belong to
only finitely many $X_s$, $s\in T$. By o-minimality, there is a bound $\ell\in
\mathbb N$ such that for every $x\in X^0$, there are at most $\ell$-many $X_s$ such
that $x\in X_s$. For every $x\in X^0$, the set $\CF^0(x)=\{X_s\scht x\in X_s\}$ is
clearly $x$-definable and finite. By DEQ each $X_s\in \CF^0(x)$ is $x$-definable,
and the finitely many $X_s\in \CF^0(x)$ can be linearly ordered as $X_{s_1},\ldots,
X_{s_\ell}$ with the ordering depending on $x$. This is a well defined finite
ordering of the sets $X_t\in \CF^0(x)$.

For each $t\in T$, we let $X_t^0=X^0_t\cap X_t$ and for $i=1\ldots, \ell$ we let the
set $X_t^i$ be the collection of all $x\in X_t^0$ such that $X_t$ is the $i$-th element in
the ordering of $\CF^0(x)$.  Note that for each $t\in T$, the sets
$X_t^1,\ldots, X^{\ell}_t$ form a partition of $X_t^0$ and that for all $X_t\neq
X_s$ and $i=1,\ldots, \ell$, we have $X_t^i\cap X_s^i=\emptyset$. Let $X_t'$ be the
complement $X_t\setminus X_t^0$. Note that the definitions of $X_t^i$ and $X_t'$ depend
only on the set $X_t$ and not on $t$. Consider the families $\CF'=\{X'_t:t\in T\}$ and
$\CF_i=\{X_t^i:t\in T\}$, with $i=1\ldots,\ell$. By our earlier observations, it is
enough to find $\CF_i$-injective maps and $\CF'$-injective maps.

Let us handle the $\CF_i$-case first.

Because $\dim X_t<k$, the dimension of each $X_t^i$ is also smaller than $k$. We can
assume by further partitioning each $X_t^i$ into cells that each $X_t^i$ is a cell
of dimension $r<k$. But then each $X_t$ is the graph of a continuous function from
some open cell $C_t^i$ in some $r$-cartesian power of $M$ into $M^{k-r}$. By
dividing into cases we may assume that all $C_t^i$ are subsets of $M^r$, the first
$r$-coordinates. So, $X^i_t$ is of the form $$\{\la\bar x,f^i_t(\bar x)\ra:\bar x\in C_t^i\}$$ where $f^i_t$ is a definable function from $C_t^i$ into $M^{k-r}$.

Let $\CF_i^{\text{proj}}=\{C_t^i:t\in T\}$. By induction,
there is an injective $\CF_i^{\text{proj}}$-function $g:T\to M^s$ for some $s$. We divide
$\CF_i^{\text{proj}}$ into two sub-families: (i) Those $C^i_t$'s for which only finitely
many distinct $X^i_t$'s project onto $C_t^i$. In this case,  the function $g$,
together with a choice of one of the finitely many $X^i_t$ which project on $C_t^i$
induce an injective map on the collection of these $X^i_t$.

The rest of the $C^i_t$'s are those for which there are infinitely many $X^i_t$'s
which project onto it. Because $T$ is one-dimensional, there are at most finitely
many such distinct $C^i_t$'s (this step fails in higher dimension). By
handling each one separately, we can assume that all $X_t^i$ project onto the same
$C_t^i$. We now fix an arbitrary point $\bar a\in C_t^i$ and define $g(t)=f^i_t(\bar
a)$. Because we defined the $X^i_t$'s to be pairwise disjoint, this is an
$\CF_i$-injective map, defined over $\bar a$.

\vspace{.3cm}

We now handle $\CF'=\{X_t':t\in T\}$ (recall that $x\in X_t'$ if and only if $x$
belongs to infinitely many $X_t$'s, so also to infinitely many $X_t'$). We let
$X'=\bigcup_t X_t'$ and claim that $\dim(X')<k$.

Assume towards contradiction that $\dim X'=k$ and consider the subset $Y$ of $T
\times M^k$ consisting of all $(t,x)$ such that $x\in X'_t$. By Claim
\ref{dimeqrel}, because each $X_t$ has
dimension $k-1$, and $\dim(T)=1$, the set $Y$ has dimension at most $k$.

Clearly, $X'$ equals the projection of $Y$ onto the second coordinate, so the
dimension of $X'$, which we assume to be $k$, equals the dimension of $Y$. But then,
the projection map from $Y$ onto $X'$ is generically finite to one, so if we pick
any generic $x\in X'$, there are at most finitely many $t$'s such that $x\in X_t'$,
contradicting our definition of $X'$.

Thus $\dim X'<k$, and we can handle $\CF'$ by induction, as pointed out earlier.
\qed

\begin{cor}\vlabel{EI1} If $\dim (X/E)=1$ then $X/E$ is in definable bijection, over parameters, with a definable
set.\end{cor}

Here is a simple corollary that we are not going to use.
\begin{cor} Every $\CM$-quotient on a definable set of
dimension two can be eliminated. Namely, if $\dim Y=2$ and $E$ is a definable
equivalence relation on $Y$ then $Y/E$ is in definable bijection (possibly, over
parameters) with a definable set.
\end{cor}
\proof We may assume that all classes have the same dimension. If the classes are
finite then we can definably choose representatives. If the classes have dimension
$1$ then $\dim(x/E)=1$ and we are done by the previous lemma. If the classes have
dimension 2 then there are only finitely many classes.\qed

\subsection{A general observation about $\CM$-quotients}

\begin{prop}\vlabel{observ-eq} Let $X/E$ be an $\CM$-quotient. Then there exists
an $\CM$-quotient $Y/E'$ which is in definable bijection with $X/E$, possibly over
parameters, such that $Y\sub I_1\times \cdots \times I_k$, for some intervals
$I_1,\ldots, I_k\sub M$, and each $I_j$ is the image of $Y/E'$ (equivalently $X/E$)
under a definable map.
\end{prop}
\proof By partitioning each equivalence class into its definably connected
components (see Lemma \ref{uniform}) and choosing one component from each class
uniformly (by DEQ), we may assume that all classes are definably connected.

For every $i=1,\ldots, n$ we let $\pi_i:M^n\to M$ be the projection onto the $i$-th
coordinate. We define $\si_1^+:X/E\to M\cup \{+\infty\}$ as follows: $\si_1^+([x])$
is the supremum of $\pi_1([x])$.  We let $J^+_1,\ldots, J^+_k$ be the definably
connected components of the image of $\si_1^+$. Similarly, we let $\si_1^-([x])\in
M\cup \{-\infty\}$  be the infimum of $\pi_1([x])$ (note that $\pi_1([x])$ is
contained in the interval $[\si^-([x]), \si^+([x])]$). Let $J^-_1,\ldots, J^-_r$ be
the definably connected components of the image of $X/E$ under $\si^-_1$. For $1\leq
i\leq k$ and $1\leq j\leq r$, we let $X_{i,j}=\{x\in X: \si^+_1([x])\in J^+_i \mbox{
and } \si^-_1([x])\in J^-_j.\}$

This is partition of $X$, which is compatible with $E$, namely if $x\in X_{ij}$ and
$x'Ex$ then also $x'\in X_{ij}$. It is clearly enough to prove the result for each
$X_{ij}$ separately. We consider two cases:

\noindent (a) $J_i^-\cup J_j^+$ is not definably connected.

If $X_{ij}$ is nonempty, then every element in $J^-_i$ is smaller than every
 element in $J^+_j$, and we can  fix an arbitrary element $a_{ij}$ with
 $J^-_i<a_{ij}<J^+_j$. By the definition of $\si^+_1$ and $\si^-_1$, and since $[x]$ is definably connected, the element
 $a_{ij}$ is contained in $\pi_1([x])$, for every $x\in X_{ij}$. For each $x\in
 X_{ij}$ we can now replace $[x]$ with $[x]\cap \pi_1^{-1}(a_{ij})$. The union of all
 these new classes, call it $X'_{ij}$, is contained in $\pi_1^{-1}(a_{ij})$. If we let $E'$
 be the restriction of $E$ to $X'_{ij}$ then $X_{ij}/E$ and $X'_{ij}/E'$ are in
 definable bijection. However, $X'_{ij}$ can be identified with a subset of
 $M^{n-1}$, so we can finish by induction.

\noindent (b) $J_i^-\cup J_i^+$ is definably connected.

In this case, we let $J_{ij}=J^-_i\cup J^+_j$. This is an interval which is the
union of two intervals, each of which is the image of $X/E$ under a definable map
({\em a priori}, each interval is the image of a subset of $X/E$, but the map can be
extended trivially to the whole  $X/E$). Because $\pi_1([x])\sub [\si^-([x]),
\si^+([x])]$ the set  $X_{ij}$ is a subset of $J_{ij}\times M^{n-1}$ and $J_{ij}$
itself can be identified with a subset of $J_i\times J_j$, possibly after naming
parameters. We can proceed by considering the projection $\pi_2$ and so on.\qed

\subsection{gp-long dimension and definable quotients}

Before our next, technical lemma we recall the following.
\begin{fact}\label{defcomp*} If $X\sub M^n$ is a definable closed and bounded set, then
$X$ contains a point $x_0$ which is invariant under any automorphism of $\CM$ which
preserves $X$ set-wise. Namely, $x_0\in  \dcl({X})$, where $X$ is now considered as
an element of $\CM^{eq}$.
\end{fact}
\proof This is the same as showing that every definable family of closed and bounded
sets has strong definable choice. For $X\sub M$, we just take $x_0$ to be $min X$,
and for $X\sub M^n$ we use induction.\qed

\begin{lemma}\vlabel{shcl-eq1} Let $X\sub M^n$ be an $A$-definable set such that $\lgdim(X)=n$.
Assume that $E$ is an $A$-definable equivalence relation in $X$, such that every
equivalence class is gp-short. Then $\dim(X/E)=n$.

Equivalently, if $a\in X$ is long-generic over $A$ then $[a]$ must be finite.
\end{lemma}
\proof First note why the result holds when $X/E$ can be eliminated. Indeed, if we
have a definable bijection between $X/E$ and a definable set $Y\sub M^r$ then we
obtain a definable surjection $g:X\to Y$ such that the preimage of every $y$ is
an $E$-class. Because every class is gp-short it follows from the dimension formula
that $\lgdim(X)=\lgdim(Y)$ and in particular $\dim Y=n$.

We now return to our setting. Let $a\in X$ be long-generic over $A$. If we show that
$\dim([a]/A)=n$ (here we view $[a]$ as an element of $\CM^{eq}$) then clearly, $\dim
(X/E)\geq n$, so we must have $\dim (X/E)=n$. Without loss of generality
$A=\emptyset$.

Using uniform cell decomposition, we partition $X$ into finitely many
$\emptyset$-definable sets $X_1,\ldots, X_m$, such that the intersection of every
class with each $X_i$ is definably connected (possibly empty). The element $a$
belongs to one of these $X_i$, in which case $\lgdim X_i=n$. We therefore may assume
that each $E$-class in $X$ is definably connected.

Let $p(x)=\tp(a/\emptyset)$. By Fact \ref{shcl-types}, there exists an $n$-long-box
$B$ with $\Cl(B)\sub p(M)$, such that $a$ is long-central in $B$. Since $[a]$ is
gp-short, we have $\Cl([a])\sub \Cl(B)\sub p(M)$, so in particular the set $\Cl([a])$
is closed and bounded in $M^n$. By Fact \ref{defcomp*}, there exists $y\in \Cl([a])$
such that $y\in \dcl(\{\Cl([a])\}).$ But clearly, $\Cl([a])$ is invariant under any
automorphism preserving $[a]$, hence $y\in \dcl(\{[a]\}).$

But then, by the dimension formula $\dim(y/\emptyset)\leq \dim([a]/\emptyset)$ and
since $y\models p$, we have $\dim(y/\emptyset)=n$ (here we use the fact that $a$ was
a generic element of $M^n$). Hence, $\dim([a]/\emptyset)=n$, so we are done.

To see that $[a]$ must be finite consider all the set $Y\sub X$ of all $x\in X$ such
that $[x]$ is infinite. By definition of dimension we must have $\dim(Y/E)<\dim
Y\leq n$. Hence, by what we have just showed, $\lgdim(Y)<n$. It follows that $Y$
cannot contain any long-generic element of $X$. \qed

%\noindent{\bf Conjecture} If $X\sub M^n$ has long dimension $n$ and $E$ is a
%definable equivalence relation on $X$ then there exists an $n$-long box $B\sub X$
%such that for every $b\in B$, $\lgdim([b])=\dim([b])$.
%\\

%Note that the above conjecture implies Lemma \ref{shcl-eq1}, because in that lemma
%we assumed that every class has long dimension $0$. If the conjecture is true then
%there is an $n$-long box $B\sub X$, on which every class also has dimension $0$,
%which in turn implies that $\dim(X/E)=n$.

\section{Interpretable groups}

We assume now that $G$ is an interpretable group (as in Definition \ref{def-int-gp}).

\subsection{One dimensional sets in interpretable groups}

\begin{defn} Let $Y\sub X$ be a definable set such that
$\dim(Y/E)=1$. Then $Y/E$ is called {\em gp-short} if it is in definable bijection
with a definable gp-short subset of $M^n$. Otherwise, we call it {\em gp-long}.
\end{defn}

\begin{theorem}\vlabel{inter-group1} Let $G=X/E$ be an interpretable group. Then every
one-dimensional subset of $G$ is gp-short.
\end{theorem}
\proof Without loss of generality, $G$ is defined over $\emptyset$. We let $I\sub G$
be a $\emptyset$-definable one-dimensional set. By Corollary \ref{EI1}, $I$ is in
bijection with a definable subset of $M$, and so we identify $I$ with this definable
subset and assume that the intersection of every $E$-class with $I$ is a singleton.
We suppose towards contradiction that $I$ is gp-long.

For every $k$ we let $f_k:I^k\to G$ be the function defined by
$f(x_1,\ldots,x_k)=x_1\cdots x_k$ (multiplying in $G$).

We take $k\geq 1$ maximal such that on some $k$-long box $B\sub I^k$ the function
$f_k$ is finite-to-one. By taking a sub-box of $B$ we may assume that $f_k$ is
injective on $B$. We assume that $B$ is definable over $\emptyset$ and let $\bar
a\in B$ be long-generic in $B$.

\begin{claim*} Let $a_{k+1}$ be long-generic in $I$ over $\bar a$. Then there
is  a $k+1$-long box $B'\sub B\times I$ containing $a'=\la \bar a,a_{k+1}\ra$, such
that $f_{k+1}(B')$ is contained in $f_{k+1}(B'\times \{a_{k+1}\})\sub f_k(B)\cdot
a_{k+1}$.
\end{claim*}

\begin{proof}[Proof of Claim] Define on $B\times I$ the equivalence relation $xE'
y$ iff $f_{k+1}(x)=f_{k+1}(y)$.  By the maximality assumptions on $k$, the union of
all finite $E$-classes must have long dimension smaller than $k+1$. Therefore, since
$\lgdim\la a,a_{k+1}\ra =k+1$, the $E'$-class of $a'=\la a,a_{k+1}\ra$ is infinite.

We claim that $\lgdim [a']>0$. Indeed, assume towards contradiction that $[a']$ is
gp-short. By Fact \ref{int-fact0}(iii), there is a formula $\psi(y)$ over
$\emptyset$ such that $\psi(a')$ holds and if $\psi(b)$ holds then $[b]$ is
gp-short. Thus, there exists, by Lemma \ref{nbox}, a $k+1$-long box $B_0 \sub
B\times I$ containing $a'$ such that for every $x\in B_0$, the $E'$-class $[x]$ is
infinite and gp-short. However, this implies that $\dim(B_0/E')<k+1$, contradicting
Lemma \ref{shcl-eq1}.

We therefore showed that the $E'$-class  of $a'$ is not gp-short. A similar argument
can show a stronger statement, namely that the definably connected component of
$[a']$ which contains $a'$, call it $[a']^0$, is also not gp-short.

Because $f_k|B$ was finite-to-one the projection of each $E'$-class on the
$k+1$-coordinate is a finite-to-one map. It follows that the image of $[a']$ under
this projection is gp-long, call it $J$.

By Fact \ref{shcl1}(3), we may replace $J$ by a possibly smaller gp-long interval
and so assume that the long dimension of $a'$ over the parameter set  $A'$ defining
$J$ is still $k+1$. Let $p(x)=\tp(a'/A')$.
 By \ref{shcl-types},
there exists a $k+1$-long box $B'\sub p(M)$, in which $a'$ is long-central. Because
$B'\sub p(M)$,  for every $x\in B'$ the projection of $[x]^0$ onto the last
coordinate contains $J$. In particular, this projection contains the point
$a_{k+1}$. This means that every $x'\in B'$ has an $E'$-equivalent element of the
form $\la x,a_{k+1}\ra$, with $x\in B$, and hence $f_{k+1}(x')=f_k(x)a_{k+1}$.

This ends the proof of the claim.
\end{proof}

Let's recall what we have so far: (i) The restriction of  $f_k$ to $B$ is an
injective map and (ii) $f_{k+1}(B')\sub f_{k}(B)\cdot a_{k+1}$.

Since $f_{k+1}(x,a_{k+1})=f_k(x)a_{k+1}$, (i) implies that the restriction of
$f_{k+1}$ to $B\times \{a_{k+1}\}$ is also injective. Therefore, we have a definable
bijection
$$\si: f_{k}(B)a_{k+1}\to B$$ (given by  $\si(y)= f_{k}^{-1}(ya_{k+1}^{-1})$) (where $a_{k+1}^{-1}$ is the
group inverse in $G$ of $a_{k+1}$).

By (ii),  we have a map from the $k+1$-long box $B'$ into $B$, defined by
$h(x_1,\ldots, x_{k+1})=\si(f_{k+1}(x_1,\ldots, x_{k+1}))$. Notice that because
$f_{k+1}$ is group multiplication and $\si$ is injective, the map $h$ is injective
in each coordinate separately. By Corollary \ref{injective}, at least one of the
intervals which make up $B'$ must be gp-short, contradicting the fact that $B'$ was a
$k+1$-long box. This shows that $I$ is gp-short, thus ending the proof of Theorem
\ref{inter-group1}.\qed

\subsection{Endowing interpretable groups with a topology}

A fundamental tool in the theory of definable groups in o-minimal structures is
Pillay's theorem, \cite{Pillay2}, on the existence of a definable basis for a group
topology on a definable group $G$, a topology which agrees with the subspace
$M^n$-topology at every generic point of $G$ (with $G\sub M^n$). Moreover, this
topology can be realized using finitely many charts, each definably homeomorphic to
an open subset of $M^k$ (with $k=\dim G$).

Here we prove an analogous result for interpretable groups but the topology we
obtain will initially not have finitely many charts.

We start with some preliminary definitions and results.
\begin{defn} A definable set $Y\sub G$ is
called {\em large in $G$} if $\dim(G\setminus Y)<\dim G$.

\end{defn}

\begin{fact} \vlabel{large} Let $G$ be an interpretable group, $Y\sub G$ a definable set over $A$. If $Y$ is
large in $G$ and $\dim G=n$ then $G$ can be covered by $\leq n+1$ translates of $Y$.
\end{fact}
\proof This is standard. We have $m=\dim(G\setminus Y)\leq n-1$. We take $g$ generic
in $G$ and $h$ generic in $G\setminus Y$ over $g$ (i.e. $\dim(h/gA)=m$). Then, by
the dimension formula, $\dim(hg^{-1}/hA)=\dim(g/A)$ and hence $hg^{-1}$ is generic
in $G$ over $A$. It follows that $hg^{-1}\in Y$ and therefore $h\in Yg$.

We showed that every element in $G$ of dimension $m$ over $A$ belongs to $Y\cup Yg$.
In particular, $\dim (G\setminus (Y\cup Yg))\leq m-1 \leq n-2$. We proceed by
induction. \qed

\subsection*{Defining the topology}
\leavevmode \bigbreak We first obtain $U_1,\ldots, U_k$ as in Claim
\ref{general-eq}. Namely, each $U_i$ is an open subset of $M^{k_i}$ and each class
in $U_i$ has dimension $d_i$ and projects homeomorphically onto the first $d_i$
coordinates. Write $x=\langle x',x''\rangle\in M^{d_i}\times M^{k_i-d_i}$, with
$x'=\pi_{d_i}(x)$. Since every $E$-class projects bijectively into $M^{d_i}$, the
set $\pi^{-1}(x')\cap U_i$ has a single representative for each $E$-class (if
$d_i=0$ then $x=x''\in M^{k_i}$). It is contained in the set $\{ x'\} \times
M^{k_i-d_i}$ and because $U_i$ is open   can be identified with an open subset of
$M^{k_i-d_i}$. Call this $x'$-definable set $U_i(x)$. We say that $V\sub U_i(x)$ is
an $M^{k_i-d_i}$-open set, if under this identification $V$ is open. We have an
obvious definable injection of each $U_i(x)$ into $G$. For $\la x',x''\ra\in U_i$,
let $[x',x'']$ denote $[\la x',x''\ra]$.

\begin{fact}\vlabel{top-facts}\leavevmode In the above setting,
\begin{enumerate}
\item For $g\in U_i$ and $x=\la x',x''\ra $ generic in the class $g$ (over the
element $g\in \CM^{eq}$ and $A\sub M^{eq}$), we have $\dim(x'/gA)=d_i$ and $x''\in
\dcl(x'g)$.

\item Assume that $x$ is generic in $U_i$ and write $x=\la x',x''\ra\in
M^{d_i}\times M^{k_i-d_i}$.
 Then $x''$ is generic in $U_i(x)$ over $x'$.

 \item If $x=\la x',x''\ra $ is generic in $U_i$, $h\in G$, and $y=\la y',y''\ra$ is generic in
 the class $h$ over the elements $x,h$ then $\dim(x''/x'y')=\dim(x''/\emptyset)=k_i-d_i$.
 \end{enumerate}
 \end{fact}
 \proof (1), (2) are immediate from the fact that each class in $U_i$ projects bijectively onto the first
 $d_i$ coordinates. For (3), let $y\in U_j$ and note that by genericity, $\dim(y/x,h)=\dim([y])=d_j$ and
 therefore by (1), $\dim(y'/x,h)=d_j$. But then, since $y'\in M^{d_j}$ we clearly
 must have $$\dim(y'/x')=\dim(y'/x)=d_j.$$

 By the dimension formula
 $$\dim(y'/x'x'')+\dim(x''/x')+\dim(x'/\emptyset)=\dim(x''/x'y')+\dim(y'/x')+\dim(x'/\emptyset).$$
%Janak fixed statement -- changed d_i to k_i-d_i
 Since $\dim(y'/x'x'')=\dim(y'/x')$, we have $\dim(x''/x'y')=\dim(x''/x')=k_i-d_i$,
 hence $x''$ is generic in $U_i(x)$ over $x'y'$.\qed

We assume from now on that for $i=1,\ldots, r$, we have $\dim (U_i/E)=k_i-d_i=\dim
G=n$ and for $i=r+1,\ldots, k$ we have $\dim (U_i/E)<\dim G$. Let $U=U_1$.

\begin{fact} \vlabel{continuous}Let $f:G\to G$ be a partial $A$-definable function. Let $x=\la x',x''\ra$ be a generic element of $U$
over $A$ and let $g=[x]$.
 Let $h=f(g)$ and choose $y=\la y',y''\ra$
a generic element of the class $h=f(g)$ over $x$.

If $y\in U_j$, for some $j=1,\ldots,k$, then there is an $Ax'y'$-definable open
$M^n$-neighborhood $V\sub U(x)$ such that for every $z\in V$ there is a unique
element $w\in U_j(y)$, with $f([x',z])=[y',w]$.
 We denote this local map, defined over $x'y'$, by $f^*$. The map $f^*$ is continuous
 at $x''$,
 as a function from an open subset of $M^n$ into $M^{k_j-d_j}$.
 \end{fact}
 \proof By Fact \ref{top-facts}(3), $x''$ is generic in $U(x)$ over $x'y'$.
 We now consider the formula $\phi(z)$, over the parameters $Ax'y'$,  which says that there is a unique element
 $w\in U_j(y)$ such that $[y',w]=f([x',z])$. The formula $\phi(z)$ holds for $x''$, which is generic in $U(x)$
 over $x'y'$.
 It follows that there exists an $M^n$-neighborhood $V\sub U(x)$ of $x''$ such that every $z\in V$ satisfies $\phi$.
 We therefore obtain a function, definable over $Ax'y'$, from $V$ into $U_j(y)$, and by genericity of $x''$,
 this function is continuous near $x''$.\qed

\begin{theorem} Let $x_0=\la x',x''\ra$ be a generic element in $U$  and
let $\{V_t:t\in T\}$ be a definable basis  of (sufficiently small)
$M^n$-neighborhoods of $x''$, all contained in $U(x_0)$.

Then the family $$\mathcal B=\{gV_t:g\in G,\,\, t\in T\}$$ is a basis for a topology
on $G$, making $G$ into a topological group.
\end{theorem}
\proof Consider $g_0=[x_0]$ and the family $\{g_0^{-1}V_t:t\in T\}$. Just like the
proof of Lemma 2.12 in \cite{Mar}, we will prove that this family forms a basis of
neighborhoods of $1\in G$, for a group-topology on $G$ whose basis is $\mathcal B$.
Indeed, it is not hard to see that $\mathcal B$ is a basis for some topology, call
it the {\em $t$-topology on $G$}. To see that this topology makes $G$ into a
topological group, we first prove:

\begin{claim}\vlabel{coincide} Let $g$ be generic in $G$ over $g_0$ and let $y=\la y',y''\ra$ be
generic in the class $g$ over $g,g_0$. Then there is an  open $M^n$-neighborhood
$W\sub U(y)$ of $y''$ and a $t$-neighborhood $V\sub  G$ of $g$ such that the
canonical embedding of $U(x)$ into $G$ induces a homeomorphism of $W$ and $V$.

Roughly speaking, we say that the $t$-topology coincides with the $M^n$-topology in
some neighborhood of $g$.
\end{claim}
\proof Consider the map $\si(h)=gg_0^{-1} h$. It is definable over the element
$gg_0$ and takes $g_0$ to $g$.  Consider the first-order formula $\phi(z)$ over
$x',y'$ which says that $z\in U(x)$ and there is a unique $w\in U(y)$ such that
$\si(z)=w$ (here we identify $U(x)$ and $U(y)$ with subsets of $G$). The formula
$\phi$ holds for $x''$. By Fact \ref{top-facts} (3), $x''$ is generic in $U(x)$ over
$x',y'$ hence there exists an $M^n$-neighborhood $V\sub U(x)$ of $x''$ such that for
every $z\in V$, $\si(z)\in U(y)$. Hence $\sigma$ defines a function from $V$ into $U(y)$
sending $x''$ to $y''$. By the genericity of $x''$, we can choose such $V$ so that
$\si$ is continuous, as a map from $M^n$ into $M^n$. Because $\si$ is invertible, we
can use the same argument to find $W\sub U(y)$ for which $\si^{-1}$ is also
continuous, as  a map into $U(y)$. Shrinking $V$ and $W$ if needed we may assume
that $\si:V\to W$ is a homeomorphism with respect to the $M^n$-topology. Since left
multiplication leaves $\CB$ invariant, $\si$ is also a homeomorphism  with respect
to the $t$-topology. It follows that the $t$-topology agrees with the $M^n$-topology
on $W$.\qed

\begin{claim}\vlabel{continuous2} Assume that $f:G\to M^d$
 is an $Ag_0$-definable partial  function  and $g$ is generic in $G$ over $Ag_0$.
 Then $f$ is continuous at $g$ (with respect to the $t$-topology on
$G$ and the standard topology on $M^d$).
 \end{claim}
\proof Choose $y=\la y',y''\ra$ generic in the class $g$ over $g, g_0$. Consider now
the map $f$, as a function from $U_j(y)$ into $M^d$. Since $y''$ is generic in
$U_j(y)$ over $Ay'g_0$, this map is continuous near $y''$, with respect to the
$M^n$-topology of $U_j(y)$, and hence, by Claim \ref{coincide}, also with respect to
the $t$-topology. \qed

We also have:
\begin{claim}\vlabel{continuous3} We fix $g_0$ as above. If $f:M^k\to G$ is a $g_0$-definable function
 then $f$ is continuous at every point $z$ generic in its domain (with respect to the $M^k$-topology and the $t$-topology),
 \end{claim}
 \proof Let $h=f(z)$ and take $g_1$ generic in $G$ over $g_0,h,z$. Instead of considering the map $f$
 we consider $\sigma(w)=g_1h^{-1} f(w)$, which sends $z$ to $g_1$. Since left
 multiplication is a homeomorphism (as it preserves the family $\mathcal B$)
 it is sufficient to show that $\sigma$ is continuous at $z$.  Using Claim \ref{coincide}, we can reduce the problem to a map from $M^k$
 into $U_j(y)$, with $[y]=[y',y'']=g_1$
 and $y$ generic in the class  $g_1$ over all parameters. After noting that
 $\dim(z/g_0,g_1,y')=\dim(z/g_0,g_1)$, so $z$ is still generic in the domain of
 $f$ over $g_0g_1y'$, the result now follows from the theory of definable maps from
 $M^k$ into $M^n$.\qed

The above results allow us to replace in many cases the $t$-topology by the
$M^n$-topology, so we can follow the arguments from \cite{Mar} and conclude in the
same way that $\CB$ defines a group-topology on $G$.\qed

Since the $t$-topology has basis for neighborhoods given by open subsets of $M^n$,
it means that, at least locally, many properties of the o-minimal topology still
hold for the $t$-topology. A straightforward claim helps here:

\begin{claim}\label{dim-gen-nbhd}
Let $Y\subseteq G$ be $A$-definable and let $g$ be generic in $Y$ over $Ag_0$. Then
for every definable $t$-open set $V\ni g$, we have $\dim(Y\cap V)=\dim(Y)$.
\end{claim}

\begin{proof} Replace $V$ with a neighborhood $W\sub V$ of $g$ which is
definable over parameters $B$, with $\dim(g/B)=\dim(g/Ag_0)$. Indeed, this is
possible to do: We may assume that $V=gg_0^{-1} V_t$ for $t\in T$, and we have
$\dim(g/g_0,A)=\dim(g/gg_0^{-1},A)$. We now replace $V_t$ by $V_s\sub V_t$, with $s$
generic in $T$ over all parameters. We therefore have $\dim
(g/gg_0^{-1},s,A)=\dim(g/g_0,A)$.

The neighborhood $W=gg_0^{-1}V_s$ is the desired neighborhood of $g$. Since the
dimension of $g$ over the parameters defining $Y\cap W$ equals $\dim(g/A)=\dim(Y)$,
we have $\dim(Y\cap W)=\dim(Y)$.
\end{proof}

\begin{fact}\leavevmode
\begin{enumerate}
\item If $Y\sub G$ is a definable set then $\dim (\Cl(Y)\setminus Y)<\dim Y$ (the
closure here is taken with respect to the $t$-topology. \item If $H$ is a definable
subgroup of $G$ then $H$ is closed in $G$.
\end{enumerate}
\end{fact}

\begin{proof} We prove (1) -- the proof of (2) is as in \cite[Corollary
2.8]{Pillay2}. Assume towards contradiction that
$\dim(\Cl(Y)\setminus Y)\geq \dim Y$. In particular, $\dim
\Cl(Y)=\dim(\Cl Y\setminus Y)$. Let $g$ be generic in both $\Cl(Y)$
and $\Cl(Y)\setminus Y$, let $h\in U$ be generic in $G$ over
$g$, and let $V$ be a neighborhood of $g$ small enough that every
element of $hg^{-1}V$ is represented in an $M^n$ neighborhood of
$h$ inside $U(h)$. Using Claim \ref{dim-gen-nbhd},
$\dim((\Cl(Y)\setminus Y)\cap V)=\dim(\Cl(Y)\setminus Y)$ and
$\dim(\Cl(Y)\cap V)=\dim(\Cl(Y))$. Translating by the
 $t$-homeomorphism $x \mapsto hg^{-1} x$, we
get $Y'=hg^{-1}(Y\cap V)$, and $\Cl(Y')\setminus
Y'=hg^{-1}((\Cl(Y)\setminus Y)\cap V)$. These sets are in
definable bijection with definable sets in an $M^n$ neighborhood
of $h$ inside $U(h)$, for which the closure operation is the
standard one, so $\dim(Y')>\dim(\Cl(Y')\setminus Y')$. However,
translation is dimension-preserving so we reach a contradiction.
\end{proof}

Although we cannot obtain at this point an atlas on $G$ with
finitely many charts, we have an approximation to it: Let
$\mathcal U$ be the disjoint union  $U_1\sqcup\cdots\sqcup U_r$.
We say that $W\sub \mathcal U$ is open if $W\cap U_i$ is open for
every $i=1,\ldots, r$. We say that $X\sub \mathcal U$ is large in
$\CU$ if $X\cap U_i$ is large in $U_i$ for every $i=1,\ldots, r$.
Note that if $W\sub \CU$ is large in $\CU$ then its image in $G$
is large in $G$.

 As we showed above, if $y=\la y',y''\ra$ is generic in $U_i$, for
$i=1\ldots,r$, then the $t$-topology agrees with the $M^n$-topology on $U(y)$, near
$y''$. This property of $y$ is first order,
 so the set $\mathcal U_0$ of all $y\in U_i$, $i=1,\ldots,r$, for which the $t$-topology
 agrees with the $M^n$-topology on $U_i(y)$ near $y$, is definable and contains $y$.
  Moreover, this set is large in $\mathcal U$.

 Let
$\pi:\CU_0\to G$ be the quotient modulo $E$. By definition of
$\CU_0$, the map $\pi:\CU_0\to G$ is  open, when $\CU_0$ is
endowed with the o-minimal topology and $G$ has the $t$ -topology.
Next, we can apply Claim \ref{continuous3} and replace $\CU_0$ by
a large open subset, call it $\CU_0$ again, on which $\pi$ is
continuous, and still open. Let $W=\pi(\CU_0)$, and note as above
that $W$ is large in $G$.
 By Fact \ref{large}, finitely many $G$-translates of $W$, $h_1W,\ldots, h_mW$,  cover $G$.  We can now
conclude:
\begin{prop} \vlabel{prop1} There are finitely many $t$-open definable sets $W_1,\ldots, W_k$ whose
union covers $G$. There exist a definable set $\CU_0$ which is  a finite disjoint
union of definable open subsets of $M^{r_i}$'s and for each $i=1,\ldots, k$ a definable
surjective map $\pi_i:\CU_0\to W_i$, such that each $\pi_i$ is continuous and open
(with respect to the o-minimal topology in the domain and the $t$-topology in the
image).
\end{prop}

As a corollary we have:
\begin{cor}
 Every definable subset of $G$ has finitely many definably connected components with respect to the $t$-topology.
\end{cor}
\proof Fix $W_1,\ldots, W_k$ as above. Take $Y\sub G$ definable,
It is enough to see that each $Y\cap W_i$ has finitely many
definably connected components. As we saw, there is a definable
and continuous map from $U_0$ onto $W_i$. The pre-image of $Y\cap
W_i$ is a definable subset of $U_0$ so has finitely many definably
connected components (with respect to the o-minimal topology). By
continuity, $Y\cap W_i$ also has finitely many components.\qed

We can now also prove, just as in the definable case (see
\cite{Pillay1}):
\begin{lemma}\vlabel{4statements} For $G$ interpretable, and $H$ a definable subgroup of $G$, the
following are equivalent:
\begin{enumerate}
\item $H$ has finite index in $G$. \item $\dim H=\dim G$. \item $H$ contains an open
neighborhood of the identity. \item $H$ is open in $G$.
\end{enumerate}
\end{lemma}

Exactly as in the case of definable groups, we can deduce the
descending chain condition:
\begin{cor} Every descending chain of definable subgroups of $G$ is finite.
\end{cor}
\subsection{Definable compactness}
{\em Below, all limits in $G$ are taken with respect to the
$t$-topology}

Our goal now is to review briefly several fundamental notions and results in the
theory of definable groups and to verify that these results hold for interpretable
$G$ as well. The intention is to collect just those results which will allow us to
prove that $G$ is definably isomorphic to a definable group.

Recall that every definable one-dimensional subset of $G$ is in definable bijection
with finitely many points and open intervals (Corollary \ref{EI1}).
\begin{defn} We say that $G$ is {\em definably compact} if for every definable
$f$ from an open interval $(a,b)$ into $G$, the limits of $f(x)$
as $x$ tends to $a$ and to $b$  exist in $G$.
\end{defn}

As in the case of definable groups (\cite{PetStein}) we have:
\begin{lemma} \vlabel{noncompact} If $G$ is not definably compact then it contains a definable,
torsion-free one-dimensional subgroup $H\sub G$.
\end{lemma}
\proof  We review briefly the proof as suggested in \cite{torfree}. Assume that the
limit $\lim_{x\to b}f(x)$ does not exist in $G$. By Lemma \ref{continuous3}, we may
assume that $f$ is continuous on $(a,b)$.
 The group $H$ is defined to be the set of all possible limits of $f(t)f(s)^{-1}$,
 as $t$ and $s$ tend to $b$ in the interval $(a,b)$. More precisely,
{\em  $H$ is the collection of all $h\in G$ such that for every $t$-neighborhood
$V\ni h$
 and every $a_0\in (a,b)$ there exist $x,x'\in (a_0,b)$ for which $f(x)f(x')^{-1}\in
 V$.}

Since $G$ has a definable basis for the $t$-topology, $H$ is definable. Note that by
o-minimality, if $h\in H$, $V\ni h$ and $a_0\in (a,b)$, then {\em for every $x'\in
(a_0,b)$} sufficiently close to $b$ there exist $x\in (a_0,b)$ with
$f(x)f(x')^{-1}\in V$.

  To see that $H$ is a subgroup, take $g,h\in H$ and show that
$gh^{-1}\in H$: Fix $V\ni gh^{-1}$ and find $t$-neighborhoods $V_1\ni g$ and $V_2\ni
h$ such that $V_1V_2^{-1}\sub V$. By the above, there exists $x'\in (a_0,b)$
sufficiently close to $b$ and there are $x_1,x_2\in (a_0,b)$ such that both
$f(x_1)f(x')^{-1}\in V_1$ and $f(x_2)f(x')^{-1}\in V_2$. It follows that
$f(x_1)f(x_2)^{-1}\in V_1V_2^{-1}\sub V$ as required, so $gh^{-1}\in H$.

The proof that $H$ has dimension at least one is similar to the proof in \cite[Lemma
3.8]{PetStein} because the identity element  of $G$ has a neighborhood $R$
homeomorphic to a rectangular  open subset of $M^n$: For every $a_0\in (a,b)$ we
have $f(a_0)f(a_0)^{-1}\in R$ and since $f(x)$ has no limit in $G$  as $x$ tends to
$b$, for all  $x'\in (a_0,b)$ close enough to $b$, we have $f(a_0)f(x')^{-1}\notin
R$, if $R$ is chosen sufficiently small. It follows that there exists $x''\in (a_0,b)$ with $f(a_0)f(x'')^{-1}\in \bd(R)$.
Because $\bd(R)$ is definably compact, as $a_0$ tends to $b$, the set of all of these
points in $\bd(R)$ has a limit point which belongs to $H$. We therefore showed that
every sufficiently small rectangular box $R\ni 1$ has a point from $H$ on its
boundary, so $\dim(H)\geq 1$.

Let's see that $\dim(H)\leq 1$: The set $D=\{\la x,x',f(x)f(x')^{-1}\ra \in
(a,b)^2\times G\}$ has dimension two and therefore its frontier $\fr(D)\sub
[a,b]^2\times G$ has dimension at most $1$. The group $H$ is contained in the
projection of $\fr(D)$ onto the $G$-coordinate.

The fact that $H$ is torsion-free is proved similarly to
\cite{PetStein}.\qed

On the definably compact side we need:

\begin{theorem}\vlabel{SDC} If  $G$ is definably compact then it has strong definable choice
(possibly over a fixed set of parameters)
for subsets of $G$ definable in $\CM^{eq}$. Namely, there is a fixed set $B\sub M$
such that if $\{Y_t:t\in T\}$ is a $\emptyset$-definable family of subsets of $G$, with $T$
definable in $M^{eq}$, then there is a $B$-definable map $\sigma:T\to G$ such that
for each $t\in T$, we have $\sigma(t)\in Y_t$, and if $Y_t=Y_s$ then
$\sigma(t)=\sigma(s)$.

Equivalently, if $Y\sub G$ is definable over $A\sub \CM^{eq}$ then $\dcl(AB)\cap
Y\neq \emptyset$.
\end{theorem}
\proof  Let us note why the two statements are indeed equivalent. Assume that we
proved strong definable choice over $B$ for families parameterized by a definable
subset of $M^{eq}$ and assume that $Y$ is definable over $a\sub \CM^{eq}$. In this
case there is a $B$-definable family of sets $\{Y_t:t\in T\}$, for some
$B$-definable set $T\sub \CM^{eq}$, with $a\in T$ and $Y_a=Y$. Strong definable
choice implies that $Y\cap \dcl(aB)\neq\emptyset$. As for the converse, assume that
we are given the family $\{Y_t:t\in T\}$ and consider the equivalence relation on
$T$ given by $s\sim t$ if and only if $Y_s=Y_t$. We now obtain a new family
$\{Y_{[t]}:[t]\in T/\sim\}$, with $Y_{[t]}=Y_t$. By our assumption,  for every
$[t]$, we have $Y_{[t]}\cap \dcl(B[t])\neq \emptyset$. But for each $t\in T$, $[t]\in
\dcl(Bt)$, and therefore $Y_{[t]}\cap \dcl(Bt)\neq \emptyset$. Strong definable choice
over $B$ follows by compactness.
\\

We now prove the theorem.
The strategy of our proof is taken from Edmundo's \cite{ed-solv}.

\begin{lemma}\vlabel{SDC1} For $G=X/E$ definably compact, let $Y\sub G$ be a definable set over  $A\sub M^{eq}$.
Then $\dcl(A)\cap \Cl(Y)\neq \emptyset$.
\end{lemma}
\proof First, note that $\Cl(Y)$ is also definably compact.

We are going to prove a slightly different statement: {\em For
every $A$-definable set $Y^*\sub M^k$ (for some $k$) and for every
$A$-definable function $g:Y^*\to G$, we have $\dcl(A)\cap
\Cl(g(Y^*))\neq \emptyset$} (to apply this statement to our case
take $Y^*\sub X$ the pre-image of $Y$ under the quotient map).

 We prove the result by induction on $\ell=\dim Y^*$. If $\ell=0$ then $Y^*$ is finite so every element of $Y^*$ is in $\dcl(A)$
 (see the earlier property DEQ) and therefore $Y\sub \dcl(A)$.

 Assume now that $\dim Y^*=\ell>0$. If $\ell=1$ then
$Y^*$ is a finite union of $A$-definable open intervals and the
restriction of $g$ to one of these gives an $A$-definable function
$g:(a,b)\to G$. Its image is either finite, so again in $\dcl(A)$
(see \cite{Pillay1}), or infinite in which case, by definable
compactness, the limit point of $g(y)$ as $y$ tends to $b$, exists
in $\Cl(g(Y^*))$ and is $A$-definable.

Assume then that $\ell>1$.  We find a projection, $\pi^*:Y^*\to M^{\ell-1}$  whose
image has dimension $\ell-1$. For every $t\in \pi^*(Y^*)$, let $Y^*_t\sub Y^*$ be
the pre-image of $t$ under $\pi^*$. By dimension considerations, we can find an
$A$-definable set $T\sub \pi^*(Y^*)$ such that for every $t\in T$, $\dim(Y^*_t)=1.$
Because $\dim Y^*_t=1<\ell$, we have, by induction, $\dcl(At)\cap \Cl(g(Y^*_t)) \neq
\emptyset$. Using compactness, we get an $A$-definable function $\sigma :T\to G$
with $\sigma(t)\in \Cl(g(Y^*_t))$ for every $t\in T$. Because $\dim T <\ell$, we can
apply induction and obtain $$\dcl(A)\cap \Cl(\si(T))\neq \emptyset.$$ But $\si(T)\sub
\Cl(g(Y^*)$, so we are done.\qed

\begin{lemma} \vlabel{SDC2} There exists a finite set $B$ and a $B$-definable neighborhood $U_0\ni
1$ in $G$ such that $G$ has strong definable choice over $B$, for definable subsets
of $U_0$.
\end{lemma}
\proof Start with a fixed neighborhood $U_0$ of $1\in G$, which we may assume is a
subset of $M^n$. The group $G$ induces on $U_0$ the structure of a local group, so
just like in \cite[Lemma 1.28]{PPS}, we may assume, by further shrinking $U_0$, that
$U_0$ is a product of intervals, each endowed with the structure of a bounded group-interval (this might require the parameter set $B$). By Lemma \ref{SDC-short}, $U_0$
has definable choice.\qed

We can now complete the proof of the theorem. Take an
$A$-definable $Y\sub G$. By Lemma \ref{SDC1}, there exists $h\in
\dcl(A)\cap \Cl(Y)$. We can now replace $Y$ by $Y_1=h^{-1}Y\cap
U_0$. The set $Y_1$ is $AB$-definable and because $h\in \Cl(Y)$,
the set $Y_1$ also non-empty. By Lemma \ref{SDC2}, we have
$\dcl(ABh)\cap Y_1\neq \emptyset$. But $h$ is in $\dcl(A)$ so we
have $\dcl(AB)\cap Y\neq \emptyset$.\qed

\subsection{Interpretable groups are definable}

\begin{theorem}\vlabel{THE THEOREM} \begin{enumerate} \item If $G$ is an interpretable group then it is
definably isomorphic, over parameters, to a definable group.

\item If $G$ is a definable group then there are generalized group-intervals
$I_1,\ldots, I_k$ and a definable injection $\si:G\to I_1\times\cdots \times I_k$.
Namely, $G$ is definably isomorphic, over parameters, to a a definable group in a
cartesian product of generalized group-intervals. We can also replace each
group-intervals $I_j$ with a one-dimensional definable group $H_j$.
\end{enumerate}
\end{theorem}
\proof We are going to prove the following statement, which incorporates both (1)
and (2): {\em Every interpretable group $G$ is definably isomorphic to a definable
group which is gp-short.}

We prove the results through several lemmas.
\begin{lemma}\vlabel{main-inter1} The result holds for $G$ definably compact.
\end{lemma}
\proof By Theorem \ref{SDC}, $G$ has strong definable choice.
 By Proposition \ref{observ-eq}, there are intervals $J_i\sub M$, $i=1\ldots, k$,
each the image of $G$ under a definable map $f_i:G\to J_i$ and a definable set
$Y\sub \Pi_i J_i$ with a definable equivalence relation $E'$ on $Y$, such that $G$
is definably bijective to $Y/E'$.

Since $G$ has strong definable choice, there are definable $1$-dimensional subsets
of $G$, call them $I_1,\ldots, I_k\sub G$, such that $f_i|I_i:I_i\to J_i$  is a
bijection.  By Theorem \ref{inter-group1}, every $I_i$ is gp-short and therefore
each $J_i$ is group-short. It follows that $\Pi_i J_i$ has strong definable choice,
so $Y/E'$ is in definable bijection with a definable subset of $\Pi_i J_i$.\qed

\begin{lemma}\vlabel{main-inter2} Assume that $H_1\sub G$ is a definable normal
subgroup, and assume that $H_1$ and $G/H_1$ are each definably isomorphic to a
definable, gp-short group. Then so is $G$.
\end{lemma}
%Janak: do we need to prove that finitely many such f's uniquely code G?
\proof As in the proof of Lemma \ref{main-inter1}, it is sufficient to prove, for
every definable map $f:G\to M$, that $f(G)$ is gp-short. Let $\pi:G\to G/H_1$ be the
quotient map. For each $y\in G/H_1$, $G_y=\pi^{-1}(y)$ is in definable bijection
with $H_1$ and therefore it is in definable bijection with a gp-short definable set.
We now write $f(G)$ as a definable union $\bigcup_{y\in G/H_1} f(G_y)$. Each set
$f(G_y)$ is gp-short and the parameter set $G/H_1$ is gp-short, so by Lemma
\ref{gp-short1}, the union $f(G)$ is gp-short.

\begin{lemma} \vlabel{main-inter3} If $G$ is abelian then $G$ is definably isomorphic to a definable
group, which is gp-short.
\end{lemma}
\proof By Lemma \ref{noncompact}, we can find a chain of definable groups
$A_1\leq\cdots\leq A_k\leq G$, such that $\dim(A_i/A_{i-1})=1$ and $G/A_k$ is
definably compact. By Corollary \ref{EI1}, each one-dimensional group is definably
isomorphic to a definable group, and  by Theorem \ref{inter-group1}, each such
group is
gp-short. So, using Lemma \ref{main-inter2}, we see that $A_k$ is definably
isomorphic to a definable, gp-short group. By Lemma \ref{main-inter1}, $G/A_k$ is
definably isomorphic to a definable (gp-short) group, so again by \ref{main-inter2},
the group $G$ is definably isomorphic to a definable gp-short group.\qed

\begin{lemma} \vlabel{main-inter4} If $G$ is definably simple (namely, $G$ is non-abelian and has no
definable non-trivial normal subgroup) and definably connected then $G$ is definably isomorphic
to a definable group which is gp-short.
\end{lemma}
\proof We  fix $U_0\ni 1$ a definable neighborhood which we may
assume to be an open subset of $M^n$.  The rest of the argument is
identical to the proof in \cite{PPS}, because all that was used
there was the basic facts about definable groups (whose analogues
we proved here for interpretable groups) together with the
existence of an $M^n$-neighborhood of the identity in $G$. To
recall, the fact that $G$ is centerless implies that we can write
$U_0$ as a cartesian product of open rectangular boxes, pairwise
orthogonal,  $R_1\times \cdots \times R_s$, where each $R_j$ is
itself a cartesian product of intervals  which are non-orthogonal
to each other (see Theorem 3.1 in \cite{PPS}). Since $G$ is
definably simple we can show that there is only one such box, so
we may write $U_0$ as a single cartesian product of pairwise
non-orthogonal group-intervals. Moreover, each interval supports
the structure of a definable real closed field and all these real
closed fields must be definably isomorphic to each other (see
\cite[Theorem 3.2]{PPS}). We now have a neighborhood $U_0$ of
$1\in G$ which we may assume to be a neighborhood of $0\in R^n$
for a definable real closed field $R$. We repeat the construction
of the Lie algebra $L(G)$ in $R$ (which only requires working in a
neighborhood of $1$), and finally embed $G$ into $GL(n,R)$ using
the adjoint embedding. Clearly, the group $GL(n,\RR)$ is
gp-short.\qed

We can now prove Theorem \ref{THE THEOREM}: We use induction on $\dim G$.
 By Lemma
\ref{main-inter2}, we may assume that $G$ is definably connected. If $G$ has a
definable infinite normal subgroup $H_1$ then, by induction, both $H_1$ and $G/H_1$
satisfy the result so again by \ref{main-inter2} we are done. So, we may assume that
no such infinite definable normal subgroup exists.

Assume then that $G$ has some finite normal subgroup. In this case, by DCC and the
connectedness of $G$, this subgroup must be contained in $Z(G)$, which by
assumption must be finite. Again, using Lemma \ref{main-inter2}, we can replace $G$
by $G/Z(G)$, which now has no definable non-trivial normal subgroup. We are left
with two possibilities: either $G$ is abelian or definably simple, so we are done by
\ref{main-inter3} and \ref{main-inter4}.

To replace each $I_j$ with a definable one-dimensional group, use Lemma
\ref{intervals-groups}. \qed

\section{Appendix: A uniform cell decomposition}
\begin{lemma}\vlabel{uniform} let $\{X_t:t\in T\}$ be a $\emptyset$-definable family of
subsets of $M^k$. Then there are finitely many $\emptyset$-definable collections
$\{X_t^i:t\in T\}$, $i=1,\ldots,m$, such that: (i) For each $i=1,\ldots, m$ and each
$t\in T$, $X_t^i\sub M^k$ is a cell. (ii) For each $t\in T$, $X_t$ is the disjoint
union of $X_t^1,\ldots X_t^m$. (iii) For each $t,s\in T$, and $i=1,\ldots, m$, if
$X_t=X_s$ then $X_t^i=X_s^i$.
\end{lemma}
\proof It is sufficient to prove: Assume that $X\sub M^n$ is definable over a
parameter set $A\sub \CM^{eq}$. Then there is a cell decomposition of $X$ that is
definable over $A$. Indeed, if we do that then we can define on the above $T$ the
equivalence relation $t\sim s$ iff $X_t=X_s$. We replace the original family with
$\{X_{[t]}:t\in T/\sim\}$, with $X_{[t]}=X_t$. If each $X_{[t]}$ has a
$[t]$-definable cell decomposition then, by compactness, there is a uniform cell
decomposition of the $X_t$'s parameterized by $T/\sim$. This easily gives us the
required result.

We now fix $X\sub M^n$
 and consider the o-minimal structure $\CM_X=\la M,<,X\ra$ with a
new predicate for $X$. Since the standard cell decomposition theorem holds in this
structure there are 0-definable, pairwise disjoint cells $C_1,\ldots, C_m$ whose
union is $X$. Each $C_i$ is clearly invariant under every automorphism of $\CM_X$.
Each $C_i$ is given by a formula $\xi_i(x)$ in the structure $\CM_X$. If we now
return to $\CM$, each $\xi_i(x)$ can be transformed into an $\CM$-formula, possibly
with parameters, which we call $\xi_i(x,a_i)$.

Each set $\xi_i(M^k,a_i)$ is invariant under any automorphism of $\CM$ which fixes
$X$ set-wise, so in particular under any automorphism which fixes $A$
point-wise.\qed

\end{document}